\documentclass[11pt]{amsart}
\usepackage{latexsym}
\usepackage{amssymb}
\usepackage{amsfonts}
\usepackage{graphicx}
\usepackage{epsf}
\usepackage[all]{xypic}
\usepackage{delarray}
\usepackage{setspace}
\usepackage{tikz}
\usetikzlibrary{arrows}

\newtheorem{theorem}{Theorem}[section]

\newtheorem{lemma}[theorem]{Lemma}
\newtheorem{proposition}[theorem]{Proposition}
\newtheorem{definition}[theorem]{Definition}
\newtheorem{example}[theorem]{Example}
\newtheorem{remark}[theorem]{Remark}

\def\bea{\begin{eqnarray*}}
\def\eea{\end{eqnarray*}}

\definemorphism{dato}\dashed \tip \notip
\definemorphism{Dato}\Dashed \tip \notip

\newcommand{\ra}{\rightarrow}

\begin{document}

\sloppy

\title[]{Structural matrix algebras, generalized flags and gradings}



\begin{abstract}
We show that a structural matrix algebra $A$ is isomorphic to the
endomorphism algebra of an algebraic-combinatorial object called a
generalized flag. If the flag is equipped with a group grading, an
algebra grading is induced on $A$. We classify the gradings
obtained in this way as the orbits of the action of a double
semidirect product on a certain set. Under some conditions on the
associated graph, all good gradings on $A$ are of this type. As a
bi-product, we obtain a new approach to compute the automorphism
group of a structural matrix algebra.
\end{abstract}

\author{F. Be\c{s}leag\u{a} and S. D\u{a}sc\u{a}lescu}
\address{University of Bucharest, Facultatea de Matematica\\ Str.
Academiei 14, Bucharest 1, RO-010014, Romania}

\address{e-mail: filoteia\_besleaga@yahoo.com, sdascal@fmi.unibuc.ro}

\date{}

\maketitle

\section{Introduction and preliminaries}

Flags and flag varieties  play a key role in Algebraic Geometry,
Representation Theory, Algebraic Groups and Combinatorics, see
\cite{lb}. In this paper we consider a more general concept of
flag and we give some applications. This new kind of flag arises
as follows.

Let $k$ be a field. A structural matrix algebra is a subalgebra of
a full matrix algebra over $k$, consisting of all matrices with
zero entries on certain prescribed positions, and allowing
anything on the remaining positions. An important number of
examples and counterexamples  use such algebras. These algebras
were called structural matrix algebras in \cite{vW}, but they had
already been considered in \cite{nowicki}. Particular examples of
structural matrix algebras are upper triangular matrix algebras,
and more generally, upper block triangular matrix algebras, which
are of fundamental importance in Linear Algebra, the study of
linear groups, the study of numerical invariants of PI algebras,
etc. A structural matrix subalgebra $A$ of $M_n(k)$ is associated
with a preorder relation (i.e. reflexive and transitive) $\rho$ on
the set $\{ 1,\ldots, n\}$; $A$ consists of all matrices
$(a_{ij})_{1\leq i,j\leq n}$ such that $a_{ij}=0$ whenever
$(i,j)\notin \rho$. We denote $A=M(\rho, k)$; in other
terminology, this is the incidence algebra over $k$ associated
with $\rho$, see \cite{spiegel}. Automorphisms of structural
matrix algebras are of interest, and the problem to understand
them gets more complicated if we take $k$ to be just a ring. If
$k$ is a field, the description of the automorphism group of
$M(\rho,k)$ was done in \cite{coelho}.

A general problem in Ring Theory is to describe and classify all
group gradings on various matrix algebras. In the case of full
matrix algebras over a field $k$, the problem was solved in
\cite{bsz} and \cite{bz} for algebraically closed $k$, and descent
theory was used in \cite{cdn} to approach the case of an arbitrary
$k$. For subalgebras of a full matrix algebra, in particular for
structural matrix algebras, it is much more complicated to
describe all gradings. The aim of our paper is to construct and
classify a certain class of gradings on structural matrix
algebras.

We show in Section \ref{structuralendo} that in the same way the
full matrix algebra $M_n(k)$ is the endomorphism group of a vector
space of dimension $n$, a structural matrix algebra $M(\rho,k)$ is
isomorphic to the endomorphism algebra of a certain
algebraic-combinatorial structure $\mathcal F$, which we call a
$\rho$-flag. In Section \ref{sectiongradedflags} we explain that
if $\mathcal F$ is additionally equipped with a $G$-grading, where
$G$ is a group, then its endomorphism algebra ${\rm End}(\mathcal
F)$ gets an induced $G$-graded algebra structure; we denote by
${\rm END}(\mathcal F)$ the obtained $G$-graded algebra. This
grading transfers to a $G$-grading on $M(\rho,k)$ via the
isomorphism mentioned above. The gradings produced in this way on
$M(\rho,k)$ are good gradings, i.e. all the matrix units present
in  $M(\rho,k)$ are homogeneous elements. It is an interesting
question whether all good gradings are obtained like this. This is
a problem of independent interest, and it can be formulated in
simple terms related to the graph $\Gamma$ associated with $\rho$:
if $G$ is a group, and on each arrow of $\Gamma$ we write an
element of $G$ as a label, such that for any two paths starting
from and terminating at the same points the product of the labels
of the arrows is the same for both paths, does the set of labels
arise from a set of weights on the vertices of $\Gamma$, in the
sense that an arrow starting from $v_1$ and terminating at $v_2$
has label $g_1g_2^{-1}$, where $g_1$ and $g_2$ are the weights of
$v_1$ and $v_2$? This problem was considered in \cite{nowicki} in
the case where $G$ is abelian, and it was showed that the answer
is positive if and only if the cohomology group $H^1(\Delta,
G)=0$, where $\Delta$ is a certain simplicial complex associated
with $\rho$. Also, for a given $\rho$, the answer to the above
question is positive for any abelian group $G$ if and only if the
homology group $H_1(\Delta)=0$. We show that the answer is
positive for any arbitrary group $G$ if and only if the normal
closure of two certain subgroups $A(\Gamma)\subseteq B(\Gamma)$ of
the free group generated by the arrows of  $\Gamma$ coincide;
$A(\Gamma)$ and $B(\Gamma)$ are defined in terms of cycles of the
un-directed graph obtained from $\Gamma$. This parallels the
result in the abelian case, where $H_1(\Delta)=B/A$ for similar
subgroups $A$ and $B$ in a free abelian group associated with
$\Gamma$. In fact we use slightly different $A$ and $B$, by
working with a different graph.

In order to classify $G$-gradings on the structural matrix algebra
$M(\rho,k)$, we first look at the isomorphisms between the
algebras ${\rm End}(\mathcal F)$ and ${\rm End}(\mathcal F')$,
where $\mathcal F$ and $\mathcal F'$ are $\rho$-flags under the
vector spaces $V$ and $V'$. An equivalence relation $\sim$  on $\{
1,\ldots, n\}$ arises from $\rho$, where $i\sim j$ whenever $i\rho
j$ and $j\rho i$. Then $\rho$ induces a partial order on the set
$\mathcal C$ of equivalence classes with respect to $\sim$. We
show in Section \ref{sectionlattice} that the ${\rm End}({\mathcal
F})$-submodules of $V$ are in a bijective correspondence with the
antichains of $\mathcal C$; let $\mathcal A(C)$ be the lattice
structure on the set of all such antichains, induced via this
bijection. Then an algebra isomorphism $\varphi:{\rm End}(\mathcal
F)\ra {\rm End}(\mathcal F')$ induces a linear isomorphism
$\gamma:V\ra V'$ which is a $\varphi'$-isomorphism for a certain
deformation of $\varphi$. The new algebra isomorphism $\varphi'$
is obtained from $\varphi$ by using a transitive function on
$\rho$ with values in $k^*$. Since $\varphi'$ is an algebra
isomorphism, $\gamma$ induces an isomorphism between the lattices
of $ {\rm End}(\mathcal F)$-submodules of $V$ and of ${\rm
End}(\mathcal F')$-submodules of $V'$, and this lattice
isomorphism reduces in fact to an automorphism of the lattice
$\mathcal A(C)$. Such an automorphism is completely determined by
an automorphism $g$ of the poset $\mathcal C$. Moreover, we
explain that $\varphi$ can be recovered from $g$, the deformation
constants producing $\varphi'$ from $\varphi$, and a matrix of
$\gamma$ in a fixed pair of bases. Thus we obtain that the set of
algebra isomorphisms from $ {\rm End}(\mathcal F)$ to ${\rm
End}(\mathcal F')$ is in a bijective correspondence with the
equivalence classes of a set involving the invertible matrices of
$M(\rho, k)$, the automorphisms of $\mathcal{C}$ preserving the
cardinality of elements, and the transitive functions on $\rho$,
with respect to an equivalence relation. In particular, if
$\mathcal F'=\mathcal F$, the automorphism group of $ {\rm
End}(\mathcal F)$ is described as a factor group of a double
semidirect product. As a bi-product, we obtain a descriptive
presentation of the automorphism group of a structural matrix
algebra. This automorphism group was computed in \cite{coelho},
and we show how the presentation in \cite{coelho} can be derived
from ours.

 For classifying $G$-gradings arising from graded
flags, we consider two $G$-graded $\rho$-flags $\mathcal F$ and
$\mathcal F'$, and we look at the isomorphisms between the graded
algebras ${\rm END}(\mathcal F)$ and ${\rm END}(\mathcal F')$.
Using the structure of isomorphisms between ${\rm End}(\mathcal
F)$ and ${\rm End}(\mathcal F')$, which we already know by now,
and adding the additional information about gradings, we obtain in
Section \ref{sectionisograded} that ${\rm END}(\mathcal F)\simeq
{\rm END}(\mathcal F')$ if and only if the connected components of
$\mathcal F$ and $\mathcal F'$ are pairwise isomorphic up to a
permutation, some graded shifts and an automorphism of $\mathcal
C$. Using this result, we show in Section \ref{sectionorbits} that
the isomorphism types of graded algebras of the form ${\rm {\rm
END}}(\mathcal F)$ are classified by the orbits of the action of a
certain group, which is a double semidirect product of a Young
subgroup of $S_n$, a certain subgroup of automorphisms of
$\mathcal C$, and $G^q$, where $q$ is the number of connected
components of $\mathcal C$, on the set $G^n$.

We use the standard terminology on gradings, see for example
\cite{nvo}.

\section{Structural matrix algebras as endomorphism  algebras}
\label{structuralendo}

Let $k$ be a field, $n$ a positive integer and $\rho$ a preorder
relation on $\{ 1,\ldots,n\}$. Let $M(\rho,k)$ be the
  structural matrix algebra
associated with $\rho$.

Let $\sim$ be the equivalence relation on $\{ 1,\ldots ,n\}$
associated with $\rho$, i.e. $i \sim j$ if and only if $i  \rho j$
and $j  \rho  i$, and let  $\mathcal{C}$  be the set of
equivalence classes. Then $\rho$ induces a partial order $\leq$ on
$\mathcal C$  defined by $\hat{i} \leq \hat{j} $ if and only if $i
\rho  j$, where $\hat{i}$ denotes the equivalence class of $i$.

For any $\alpha \in \mathcal{C}$, let $m_{\alpha}$ be the number
of elements of $\alpha$.

\begin{definition} \label{de}
A $\rho$-flag is an $n$-dimensional vector space $V$ with a family
$(V_{\alpha})_{\alpha \in \mathcal{C}}$ of subspaces such that
there is a basis $B$ of $V$ and  a partition $B=\displaystyle
\bigcup_{\alpha \in \mathcal{C}} B_{\alpha}$ with the property
that $|B_\alpha|=m_\alpha$ and $\displaystyle\bigcup_{\beta \leq
\alpha} B_{\beta}$ is a basis of $V_{\alpha}$ for any $\alpha \in
\mathcal{C}$. If $\mathcal{F}=(V,( V_{\alpha})_{\alpha \in
\mathcal{C}})$ and $\mathcal{F}'=(V',( V'_{\alpha})_{\alpha \in
\mathcal{C}})$ are $\rho$-flags, then a morphism of $\rho$-flags
from $\mathcal{F}$ to $\mathcal{F}'$ is a linear map $f:V \ra V'$
such that $f(V_{\alpha}) \subset V'_{\alpha}$ for any $\alpha \in
\mathcal{C}$.

\end{definition}

If $\rho$ is just the usual ordering relation on $\{ 1, \dots , n
\} $ (which corresponds to $M(\rho,k)$ being  the algebra of upper
triangular matrices), then a $\rho$-flag is just an usual flag on
an $n$-dimensional vector space.

More generally, if $\rho$ is such that $\mathcal{C}=\{
\alpha_1,\ldots ,\alpha_r\}$ is totally ordered, say
$\alpha_1<\ldots <\alpha_r$, and $|\alpha_i|=m_i$ for any $1\leq
i\leq r$, then a $\rho$-flag is a flag of signature $(m_1,\ldots
,m_r)$, and $M(\rho ,k)$ is the algebra
$$\left(
\begin{array}{cccc}
M_{m_1}(k)&M_{m_1,m_2}(k)&\ldots&M_{m_1,m_r}(k)\\
0&M_{m_2}(k)&\ldots&M_{m_2,m_r}(k)\\
\ldots&\ldots&\ldots&\ldots\\
0&0&\ldots&M_{m_r}(k)
\end{array}
\right)$$  of upper block triangular matrices, with diagonal
blocks of size $m_1,\ldots ,m_r$.

For any $i,j \in \{ 1, \dots , n \}$ with $i  \rho  j$, let
$e_{ij}$ be the matrix in $M_n(k)$ having $1$ on the
$(i,j)$-position and $0$ elsewhere. The set of all such $e_{ij}$'s
is a basis of $M(\rho,k)$. Now we present $M(\rho,k)$ as an
algebra of endomorphisms.

\begin{proposition}  \label{isoEndMat}
Let $\mathcal{F}=(V,(V_{\alpha})_{\alpha \in \mathcal{C}})$ be a
$\rho$-flag. Then the algebra ${\rm End}(\mathcal{F})$ of
endomorphisms of $\mathcal{F}$ (with the map composition as
multiplication) is isomorphic to $M(\rho,k)$.
\end{proposition}

\begin{proof}

Let $B$ be a basis of $V$ as in Definition \ref{de}. We can choose
some set of indices such that $B= (v_i)_{1\leq i\leq n}$ and for
any $\alpha \in {\mathcal C}$ the basis $B_\alpha$ of $V_\alpha$
is just $\{ v_i| i\in \alpha\}$. If $i,j \in \{ 1, \ldots , n \}$,
then $i \rho j$ if and only if $i \in \alpha$ and $j \in \beta$
for some $\alpha, \beta \in \mathcal{C}$ with $\alpha \leq \beta$;
in this case let $E_{ij} \in {\rm End}(\mathcal{F})$ be defined by
$E_{ij}(v_t)=\delta_{jt}v_i$ for any $t$ ($\delta_{ij}$ is
Kronecker's delta). Then it is easy to see that
$E_{ij}E_{pq}=\delta_{jp}E_{iq}$ for any $i,j,p,q$ and $\{E_{ij} \
| \ i  \rho  j \}$ is a basis of ${\rm End}(\mathcal{F})$. Hence
the linear map $E_{ij} \mapsto e_{ij}$ for any $i,j$ with $i  \rho
j$, is an algebra isomorphism between ${\rm End}(\mathcal{F})$ and
$M(\rho,k)$.
\end{proof}

We associate with $\rho$ another useful combinatorial object. Let
$\Gamma=(\Gamma_0,\Gamma_1)$ be the graph whose set $\Gamma_0$ of
vertices is the set $\mathcal{C}$ of equivalence classes. The set
$\Gamma_1$ of arrows is constructed as follows: if $\alpha, \beta
\in \mathcal{C}$, there is an arrow $a$ from $\alpha$ to $\beta$
(we write $s(a)=\alpha$, $t(a)=\beta$) if $\alpha < \beta$ and
there is no $\gamma \in \mathcal{C}$ with $\alpha < \gamma <
\beta$. Clearly, if $\alpha, \beta \in \mathcal{C}$, then $\alpha
\leq \beta$ if and only if there is a path in $\Gamma$ starting
from $\alpha$ and ending at $\beta$ (recall that paths of length
zero are just vertices of $\Gamma$). We denote by $\Gamma^u$ the
undirected graph obtained from $\Gamma$ when we forget the
orientation of arrows.

\section{The lattice of ${\rm End}(\mathcal{F})$-submodules of $V$
}\label{sectionlattice}

Let $\mathcal{F}= (V_{\alpha})_{\alpha \in \mathcal{C}}$ be a
$\rho$-flag on the space $V$. The aim of this section is to
describe the lattice $\mathcal{L}_{{\rm End}(\mathcal{F})} (V)$ of
submodules of the ${\rm End}(\mathcal{F})$-module $V$, where the
action of ${\rm End}(\mathcal{F})$ on $V$ is just the restriction
of the usual ${\rm End}(V)$-action on $V$. We also determine the
automorphisms of this lattice.

If $\mathcal{D}$ is a subset of $\mathcal{C}$, we denote by
$V_{\mathcal{D}}= \displaystyle \sum_{\alpha \in \mathcal{D}}V_{
\alpha}$. By convention $V_{\varnothing}=0$.

\begin{proposition}
The ${\rm End}(\mathcal{F})$-submodules of $V$ are the subspaces
of the form $V_{\mathcal{D}}$, where $\mathcal{D}$ is a subset of
$\mathcal{C}$.
\end{proposition}

\begin{proof}

Since $V_{\mathcal{D}}= \displaystyle \sum_{\alpha \in
\mathcal{D}}V_{\alpha}=\displaystyle \sum_{i=1}^n \{V_i \ | \
\mbox{there exists} \ \alpha \in \mathcal{D} \text{ such that }
\hat{i} \leq \alpha  \}$, and $E_{pq}  v_i = \delta_{q,i} v_p$ for
any $p,q$ with $\hat{p} \leq \hat{q}$, and any $i$, it is clear
that $V_{\mathcal{D}}$ is an ${\rm End}(\mathcal{F})$-submodule of
$V$.

Conversely, let $X$ be an ${\rm End}(\mathcal{F})$-submodule of
$V$. If $v=\sum_{i=1}^n a_i v_i \in X-\{0\}$, and $a_{i_0}\neq 0$
for some $i_0$, then for any $j$ with $\widehat{j} \leq
\widehat{i_0}$ we have $E_{j i_0}v=a_{i_0}v_j$, so $v_j \in X$,
thus $V_{\widehat{i_0}} \subset X$. If $\mathcal{D}=\{
\widehat{i_0} \ | \ \mbox{there exists} \ v=\sum_{i=1}^n a_i v_i
\in X-\{0\} \text{ with } a_{i_0} \neq 0\}$, we obtain that
$V_{\mathcal{D}}= \displaystyle \sum_{\alpha \in
\mathcal{D}}V_{\alpha } \subset X$. Obviously $X \subset
V_{\mathcal{D}}$, so then $X=V_{\mathcal{D}}$.

\end{proof}

\begin{remark} \label{remarkisolattices}
If $\mathcal{D} \subset \mathcal{C}$ and $\mathcal{D}_{max}$ is
the set of maximal elements of $\mathcal{D}$ (with respect to the
partial order of $\mathcal{C}$), it is clear that
$V_{\mathcal{D}}=V_{\mathcal{D}_{max}}$. As $\mathcal{D}_{max}$ is
an antichain in $\mathcal{C}$ (i.e. a subset of $\mathcal{C}$
whose any two different elements are not comparable with respect
to $\leq$), and for any antichain $\mathcal{D}$ we have
$\mathcal{D}_{max}=\mathcal{D}$, we conclude that the ${\rm
End}(\mathcal{F})$-submodules of $V$ are $V_{\mathcal{D}}$ with
$\mathcal{D}$ an antichain in $\mathcal{C}$.

If we denote by $\mathcal{A}(\mathcal{C})$ the set of all
antichains of $\mathcal{C}$, we have a bijection between
$\mathcal{A}(\mathcal{C})$ and $\mathcal{L}_{
{\rm End}(\mathcal{F})}(V)$, given by $\mathcal{D} \mapsto
V_{\mathcal{D}}$. Its inverse is $X \mapsto \mathcal{D}_{max}$,
where $\mathcal{D}$ is a subset of $\mathcal{C}$ such that
$X=V_{\mathcal{D}}$. This bijection induces a lattice structure on
$\mathcal{A}(\mathcal{C}) $  from the lattice $\mathcal{L}_{
{\rm End}(\mathcal{F})}(V)$.  Since for $\mathcal{D}, \mathcal{E} \in
\mathcal{A}(\mathcal{C})$ we have
$$V_{\mathcal{D}}\cap V_{\mathcal{E}}= \sum_{\alpha \in \mathcal{C}} \{ V_{\alpha}  \ |
 \mbox{there exist }   \beta_1 \in \mathcal{D}, \beta_2 \in
\mathcal{E} \mbox{ such that  } \alpha \leq \beta_1 \mbox{ and }
\alpha \leq \beta_2 \}$$ and
$$V_{\mathcal{D}}+ V_{\mathcal{E}}= V_{\mathcal{D} \cup \mathcal{E}},$$
we see that the infimum and the supremum in
$\mathcal{A}(\mathcal{C})$ are given by
$$\mathcal{D} \land \mathcal{E} = \{ \alpha \in \mathcal{C} \ | \mbox{there exist }  \beta_1 \in \mathcal{D}, \beta_2 \in \mathcal{E}
\mbox{ such that  } \alpha \leq \beta_1 \mbox{ and } \alpha \leq
\beta_2 \}_{max}, $$
$$ \mathcal{D} \lor \mathcal{E} = (\mathcal{D} \cup \mathcal{E})_{max} $$
for any $\mathcal{D} , \mathcal{E} \in \mathcal{A}(\mathcal{C})$.
Note that $(\mathcal{D} \cup \mathcal{E})_{max}$ may be strictly
contained in $\mathcal{D} \cup \mathcal{E}$, since $\mathcal{D}
\cup \mathcal{E}$ is not necessarily an antichain.

The partial order relation on $\mathcal{A}(\mathcal{C})$  is
$$\mathcal{D} \leq \mathcal{E} \iff V_{\mathcal{D}} \subset V_{\mathcal{E}} \iff  \mbox{ for any } \alpha \in \mathcal{D} \mbox{ there exists }
\beta \in \mathcal{E} \mbox{ such that  } \alpha \leq \beta.$$

Note that this is not the inclusion, thus
$\mathcal{A}(\mathcal{C})$ is not in general a sub-poset of the
power set $(\mathcal{P}(\mathcal{C}), \subseteq)$.
\end{remark}

The next result describes the automorphisms of the lattice
$\mathcal{A}(\mathcal{C})$.

\begin{proposition} \label{autolattices}
If $g$ is an automorphism of the poset $(\mathcal{C}, \leq)$, then
the map $f_g : \mathcal{A}(\mathcal{C}) \xrightarrow{}
\mathcal{A}(\mathcal{C})$, $f_g(\mathcal{D})=g(\mathcal{D})=\{
g(\alpha) \ | \ \alpha \in \mathcal{D}\}$ is an automorphism of
the lattice $\mathcal{A}(\mathcal{C})$. Moreover, for any lattice
automorphism $f$ of $\mathcal{A}(\mathcal{C})$ there exists an
automorphism $g$ of the poset $(\mathcal{C},\leq)$ such that
$f=f_g$.
\end{proposition}
\begin{proof}
It is straightforward to check the first part. Now let $f$ be a
lattice automorphism of $\mathcal{A}(\mathcal{C})$. Let
$\mathcal{L}_0(\mathcal{C})$ be the set of all minimal elements of
$\mathcal{C}$, $\mathcal{L}_1(\mathcal{C})$ be the set of all
minimal elements of
$\mathcal{C}\setminus\mathcal{L}_0(\mathcal{C})$ (or equivalently,
elements of height $1$ in $\mathcal{C}$)  and recurrently we
define $\mathcal{L}_h (\mathcal{C})$ for any  $h$, as the set of
all minimal elements of $\mathcal{C}\setminus(\mathcal{L}_0
(\mathcal{C}) \cup \dots \cup \mathcal{L}_{h-1}(\mathcal{C}))$ (or
equivalently, elements of height $h$ in $\mathcal{C}$).

The set $\mathcal{L}_0$ of minimal elements of
$\mathcal{A}(\mathcal{C})$ consists of all singletons $\{
\alpha\}$, with $\alpha \in \mathcal{L}_0 (\mathcal{C})$. As $f$
is a lattice automorphism, we have $f(\mathcal{L}_0)=
\mathcal{L}_0$, and this induces a bijection $g_0 : \mathcal{L}_0
(\mathcal{C}) \xrightarrow{} \mathcal{L}_0(\mathcal{C})$. We note
that $g_0$ (or $f|_{\mathcal{L}_0}$) uniquely determines the value
of $f$ at any non-empty $\mathcal{D} \subset \mathcal{L}_0$.

Next $\mathcal{A}(\mathcal{C}) \setminus
\mathcal{P}(\mathcal{L}_0)$ is a poset with the order inherited
from $\mathcal{A}(\mathcal{C})$, such that for any two elements
their supremum exists. Moreover, $f$ induces by restriction an
isomorphism of posets $f : \mathcal{A}(\mathcal{C})\setminus
\mathcal{P}(\mathcal{L}_0) \xrightarrow{} \mathcal{A}(\mathcal{C})
\setminus \mathcal{P}(\mathcal{L}_0)$ (there is no harm if we also
denote it by $f$), such that $f(x \lor y)=f(x) \lor f(y)$ for any
$x,y$. The set $\mathcal{L}_1$ of minimal elements of
$\mathcal{A}(\mathcal{C})\setminus\mathcal{P}(\mathcal{L}_0)$
consists of all singletons $\{\alpha\}$, where $\alpha$ is minimal
in $\mathcal{C}\setminus\mathcal{L}_0(\mathcal{C})$, i. e. $\alpha
\in \mathcal{L}_1 (\mathcal{C})$. Since $f(\mathcal{L}_1)=
\mathcal{L}_1$, $f$ induces a bijection $g_1
:\mathcal{L}_1(\mathcal{C})
\xrightarrow{}\mathcal{L}_1(\mathcal{C})$, and $g_0$ and $g_1$
uniquely determine the value of $f$ at any antichain $\mathcal{D}$
with $\mathcal{D} \subset \mathcal{L}_0(\mathcal{C}) \cup
\mathcal{L}_1(\mathcal{C})$.

We continue recurrently, by considering for any $h$ the set
$\mathcal{L}_h$ of minimal elements in $\mathcal{A}(\mathcal{C})
\setminus \mathcal{P}(\mathcal{L}_0 \cup \dots \cup
\mathcal{L}_{h-1})$. This consists of all singletons $\{
\alpha\}$, where $\alpha \in \mathcal{L}_h(\mathcal{C})$. As
above, $f$ induces an automorphism of the poset
$\mathcal{A}(\mathcal{C})\setminus \mathcal{P}(\mathcal{L}_0 \cup
\dots \cup \mathcal{L}_{h-1})$, hence a bijection $g_h :
\mathcal{L}_h(\mathcal{C}) \xrightarrow{} \mathcal{L}_h
(\mathcal{C})$. As $\displaystyle \bigcup_h
\mathcal{L}_h(\mathcal{C})=\mathcal{C}$, the coproduct (i.e.
disjoint union) $g$ of all $g_h$'s  is an automorphism of the
poset $\mathcal{C}$, and it is clear that $f=f_g$.
\end{proof}

\section{Isomorphisms between endomorphism algebras of
flags}\label{sectionisomorphisms}

We consider the set
$${\rm Aut}_0(\mathcal{C},\leq)=\{ g \in
{\rm Aut}(\mathcal{C}, \leq) \ | \ m_{\alpha}=m_{g(\alpha)} \text{
for any } \alpha \in \mathcal{C} \},$$ which is a subgroup of
${\rm Aut}(\mathcal{C})$. For any $g\in {\rm Aut}_0(\mathcal{C})$
we define a bijection $\tilde{g}: \{1,\ldots ,  n \} \rightarrow
\{1, \ldots ,  n \}$ as follows: if $\alpha=\{i_1,\ldots , i_r \}$
with $i_1<\ldots <i_r$, and $g(\alpha)=\{j_1, \ldots ,  j_r \}$
with $j_1<\ldots <j_r$, then $\tilde{g}(i_1)=j_1,\ldots ,
\tilde{g}(i_r)=j_r$. If $g,h \in {\rm Aut}_0(\mathcal{C})$, we
clearly have $\widetilde{hg}=\tilde{h}\tilde{g}$, thus $g \mapsto
\tilde{g}$ is an embedding of ${\rm Aut}_0(\mathcal{C})$ into the
symmetric group $S_n$.\\

If  $A \in M_n(k)$, let $A^g \in M_n(k)$ be the matrix whose
$(i,j)$-entry is the $(i,\tilde{g}(j))$-entry of $A$, and let $^gA
\in M_n(k)$ be the matrix whose $(i,j)$-entry is the
$(\tilde{g}(i),j)$-entry of $A$. If $g,h \in {\rm
Aut}_0(\mathcal{C})$, then $$(A^g)^h=A^{gh},\; ^h(^gA)=\,
{^{gh}A},\; (^gA)^h=\, ^g(A^h)$$ for any $A \in M_n(k)$. Also
$$(AB)^g=AB^g,\; ^g(AB)= (^gA) B, \; (A^g) B=A ( ^{g^{-1}}B)$$ for
any $A,B \in M_n(k)$ and any $g \in {\rm Aut}_0(\mathcal{C})$.  It
follows that if $A$ is
invertible, then so is $A^g$, and its inverse is $^g(A^{-1})$.\\

Let $\mathcal{T}=\{(a_{ij})_{i  \rho  j} \subset k^{*} \ | \
a_{ij} a_{jr}=a_{ir} \text{ for any  } i,  j,  r \text{ with } i
\rho  j,  j  \rho  r\}$. Using the terminology of Section
\ref{sectiongradedflags}, $\mathcal{T}$ can be identified with the
set of transitive $k^{*}$-valued functions on $\rho$.
Multiplication on positions (i.e. pointwise multiplication of
functions) makes
$\mathcal{T}$ a group.\\

Let $\mathcal{F}$ be a $\rho$-flag on the space $V$,
and $\mathcal{F}'$ be a $\rho$-flag on the space $V'$.
We keep the notation of Section \ref{structuralendo} for
$\mathcal{F}$, the basis of $V$, and the associated $E_{ij}$'s.
Thus we fix a basis $(v_i)_i$ of $V$ such that $\{v_i \ | \
\hat{i} \leq \alpha \}$ is a basis of $V_{\alpha}$, and similarly
a basis $(v'_{i})_i$ of $V'$ such that $\{v'_i \  | \ \hat{i} \leq
\alpha\}$ is a basis of $V'_{\alpha}$ for any $\alpha$; let
$(E'_{ij})_{i{ \rho} j}$ be
the basis of ${\rm End}(\mathcal{F'})$ associated with $(v_i')_i$. \\

Define
$$F: U(M(\rho,k)) \times {\rm Aut}_0(\mathcal{C}) \times \mathcal{T} \xrightarrow{} {\rm Iso}_{alg}({\rm End}(\mathcal{F}),{\rm End}(\mathcal{F'}))$$
as follows. If  $A \in U(M(\rho,k))$, $g \in {\rm
Aut}_0(\mathcal{C})$ and $(a_{ij})_{i  \rho  j} \in \mathcal{T}$,
let $(w_i)_{1\leq i\leq n}$ be the basis of $V'$ defined by
$$(\ldots , w_i , \ldots )=(\ldots , v'_i, \ldots)\ A^g.$$
Then for any $i, j$ with $i  \rho  j$ let $F_{ij} \in {\rm
End}(V')$ be such that $F_{ij}(w_j)=a_{ij}w_i$ for any $i \rho j$,
and $F_{ij}(w_r)=0$ for any $r\neq j$. Clearly
$F_{ij}F_{jr}=F_{ir}$ for $i \rho j$, $j
 \rho  r$.

\begin{lemma}
With the above notation, let $A=(\lambda_{pq})_{p,q}$  and
$A^{-1}=(\overline{\lambda}_{pq})_{p,q}$. Then
$$F_{ij}=a_{ij}
\displaystyle \sum_{\substack{s  \rho  \tilde{g}(i)
\\ \tilde{g}(j)  \rho  t }} \lambda_{s  \tilde{g}(i)}
\overline{\lambda}_{\tilde{g}(j) t} E'_{st}.$$ In particular
$F_{ij} \in {\rm End}(\mathcal{F}')$.
\end{lemma}

\begin{proof}
Since $(\ldots , v'_i, \ldots )=(\ldots , w_i,\ldots
)(A^g)^{-1}=(\ldots, w_i, \ldots ) \ ^g(A^{-1})$, we see that
$v'_t= \displaystyle \sum_p \overline{\lambda}_{ \tilde{g}(p)  t}
w_p$ for any $t$. Then \bea
F_{ij}(v'_t)&=&\sum_{p}\overline{\lambda}_{\tilde{g}(p) t}F_{ij}(w_p)\\
&=&a_{ij}\overline{\lambda}_{\tilde{g}(j)
t}w_i\\
&=&a_{ij} \sum_s \overline{\lambda}_{\tilde{g}(j)  t} \lambda_{s
\tilde{g}(i)}v'_s. \eea The coefficient of $v'_s$ in the last sum
may be non-zero only if $\hat{s} \leq
\widehat{\tilde{g}(i)}=g(\hat{i})$ and
$g(\hat{j})=\widehat{\tilde{g}(j)} \leq \hat{t}$, thus it is zero
unless $\hat{s} \leq \hat{t}$. This shows that $ \ \ $
$F_{ij}=a_{ij} \displaystyle \sum_{\substack{s  \rho \tilde{g}(i)
\\ \tilde{g}(j)  \rho  t  }} \lambda_{s \tilde{g}(i)}
\overline{\lambda}_{\tilde{g}(j)  t}E'_{st}$.

\end{proof}

Now define $F(A, g, (a_{ij})_{i  \rho  j})=\varphi$, where
$\varphi : {\rm End}(\mathcal{F}) \xrightarrow{} {\rm
End}(\mathcal{F'})$ is the linear map such that $\varphi
(E_{ij})=F_{ij}$ for any $i  \rho  j$. Clearly $\varphi$ is an
algebra isomorphism.

\begin{proposition} \label{surjective}
$F$ is surjective.
\end{proposition}
\begin{proof}
Let $\varphi:{\rm End}(\mathcal{F}) \xrightarrow{} {\rm End}(\mathcal{F'})$ be
an algebra isomorphism. Denote $F_{ij}=\varphi(E_{ij})$ for any
$i,  j$ with $i  \rho  j$. Then $(F_{ij})_{i
\rho  j}$ is a basis of ${\rm End}(\mathcal{F'})$,
$F_{ij}F_{jr}=F_{ir}$ for any $i,  j,  r$ with $i \rho
 j$ and $j  \rho  r$, and $(F_{ii})_i$ is a complete
set of orthogonal idempotents of ${\rm End}(\mathcal{F}')$.

It is easy to see that $V'= \displaystyle \bigoplus_{i=1}^n Q_i$,
where $Q_i={\rm Im} F_{ii} \neq 0$, so $\text{dim } Q_i=1$ for any $i$. Choose
some non-zero $w_i \in Q_i$ for any $i$. Then $(w_i)_i$ is a basis
of $V'$. Since $F_{ij}=\varphi(E_{ij})\neq 0$,
$F_{ij}(Q_r)=F_{ij}F_{rr}(V')=0$ for any $r \neq j$, and
$F_{ij}(Q_j)=F_{ii}F_{ij}(Q_j) \subset Q_i$, we see that
$F_{ij}(w_j)=a_{ij}w_i$ for some $a_{ij} \in k^*$. Since
$F_{ij}F_{jr}=F_{ir}$, we must have $a_{ij}a_{jr}=a_{ir}$ for any
$i,  j,  r$ with $i  \rho  j$ and $j  \rho
r$. In particular, $a_{ii}=1$ for any $i$.

Let $\gamma : V \xrightarrow{} V'$ be the linear isomorphism such
that $\gamma (v_i)=w_i$ for any $i$. Regard $V$ as a left
${\rm End}(\mathcal{F})$-module, and $V'$ as a left
${\rm End}(\mathcal{F}')$-module with the usual action of the
endomorphism algebra. If $i  \rho  j$ we have
$$\gamma(E_{ij}  v_t)=\gamma (\delta_{jt}v_i)=\delta_{jt}w_i, \text{ and }$$
$$F_{ij} \gamma (v_t)= F_{ij}  w_t= \delta_{jt}a_{ij}w_i$$
for any $t$. We get
\begin{equation} \label{R1}
\varphi(E_{ij})  \gamma(v_t)=a_{ij} \gamma(E_{ij}  v_t) \text{ for
any } i  \rho  j \text{ and  any }t.
\end{equation}
 Thus in general $\gamma$ is not a
$\varphi$-isomorphism, the obstruction being the scalars $a_{ij}$.
However, $\gamma$ is a $\varphi'$-isomorphism for a deformation
$\varphi'$ of $\varphi$. Indeed, the linear map $\theta:
{\rm End}(\mathcal{F}) \xrightarrow{} {\rm End}(\mathcal{F})$ defined by
$\theta(E_{ij})= a_{ij}^{-1}E_{ij}$ for any $i  \rho
j$, is an algebra automorphism, and (\ref{R1}) shows that
\begin{equation} \label{R2}
\gamma(E_{ij}  v_t)= (\varphi \theta)(E_{ij}) \gamma(v_t) \text{
for any  }i  \rho  j, \text{ and any }t. \end{equation} Thus
$\gamma$ is a $\varphi'$-isomorphism, where $\varphi'=\varphi
\theta: {\rm End}(\mathcal{F}) \xrightarrow{}{\rm
End}(\mathcal{F'})$ is also an algebra isomorphism. Then the
lattice of ${\rm End}(\mathcal{F})$-submodules of $V$ is
isomorphic to the lattice of ${\rm End}(\mathcal{F'})$-submodules
of $V'$ via the map
$$\overline{\gamma}: \mathcal{L}_{{\rm End}(\mathcal{F})}(V) \xrightarrow{} \mathcal{L}_{{\rm End}(\mathcal{F'})} (V'), $$
$$ \overline{\gamma}(X)=\gamma(X) \text{ for any  } {\rm End}(\mathcal{F}) \text{-submodule } X \text{ of  }V.$$
By Remark \ref{remarkisolattices}, there is an isomorphism of
lattices $\Phi : \mathcal{L}_{{\rm End}(\mathcal{F})}(V)
\xrightarrow{}\mathcal{A}(\mathcal{C})$, given by
$\Phi(X)=\mathcal{D}_{max}$, where $X=V_{\mathcal{D}}$; its
inverse is $\Phi^{-1}(\mathcal{D})=V_{\mathcal{D}}$ for any
$\mathcal{D} \in \mathcal{A}(\mathcal{C})$. Similarly, there is an
isomorphism of lattices $\Phi':\mathcal{L}_{{\rm End}(\mathcal{F'})}(V')
\xrightarrow{}\mathcal{A}(\mathcal{C})$. Let $f :
\mathcal{A}(\mathcal{C}) \xrightarrow{}\mathcal{A}(\mathcal{C})$
be the isomorphism of lattices such that the diagram
$$\xymatrix{  \mathcal{L}_{{\rm End}(\mathcal{F})}(V) \ar[d]_{\Phi}^{\sim} \ar[r]^{\overline{\gamma}}_{\sim}  & \mathcal{L}_{{\rm End}(\mathcal{F'})}(V')  \ar[d]^{\Phi'}_{\sim} \\ \mathcal{A}(\mathcal{C}) \ar@{.>}[r]^{f} & \mathcal{A}(\mathcal{C}) } $$
is commutative. By Proposition \ref{autolattices}, $f=f_g$ for
some automorphism $g$ of the poset $(\mathcal{C}, \leq)$. Then for
any $\alpha \in \mathcal{C}$ \bea
\gamma(V_{\alpha})&=&\overline{\gamma}(V_{\alpha})\\
&=&(\Phi'^{-1}f_g \Phi)(V_{\alpha})\\
&=&(\Phi'^{-1}f_g)(\{\alpha\})\\
&=& \Phi'^{-1}(\{g(\alpha) \})\\
&=&V'_{g(\alpha)}.\eea This shows that $\text{dim } V_{\alpha}=\text{dim }
V'_{g(\alpha)}=\text{dim } V_{g(\alpha)}$ for any $\alpha \in
\mathcal{C}$. Since $\text{dim }V_{\alpha}=\displaystyle \sum_{\beta \leq
\alpha}m_{\beta}$, we see by induction on the length of $\alpha$
that $m_{\alpha}=m_{g(\alpha)}$ for any $\alpha \in \mathcal{C}$, so
$g \in {\rm Aut}_0(\mathcal{C})$.

 Let $M$ be the matrix of $\gamma$ in
the bases $(v_i)_i$ and $(v'_i)_i$. Since
$\gamma(V_{\alpha})=V'_{g(\alpha)}$, $M$ may have non-zero entries
only on positions $(i,j)$ with $\hat{i} \leq g(\hat{j})$, i.e. on
blocks $(\alpha,\beta)$ with $\alpha \leq g(\beta)$. We see that
$A=M^{g^{-1}} \in M(\rho, k)$ and $M=A^g$. Moreover,
 $A$ is invertible and $A^{-1}=\, ^{g^{-1}}(M^{-1})$.

 Now we have $\varphi=F(A, g, (a_{ij})_{i  \rho  j})$,
 and this ends the proof.

\end{proof}

We consider the relation $\approx$ on $U(M(\rho,k))
\times {\rm Aut}_0(\mathcal{C}) \times \mathcal{T}$ defined by
$(A,g,(a_{ij})_{i  \rho  j}) \approx (B, h, (b_{ij})_{i
 \rho  j})$ if and only if $g=h$ and there exist
$d_1, \ ..., \ d_n \in k^*$ such that $a_{ij}b^{-1}_{ij}=d_i
d^{-1}_j$  for any $i  \rho j$, and $B^g=  A^g \, {\rm diag}(d_1,
\ldots , \ d_n)$ (or equivalently $AB^{-1}=\,
^{g^{-1}}(D^{-1})^{g^{-1}}$ where $D={\rm diag}(d_1,\ldots , \ d_n
)$); here ${\rm diag}(d_1,\ldots , \ d_n )$ denotes the diagonal
matrix with diagonal entries $d_1,\ldots,d_n$. This is clearly an
equivalence relation.

\begin{theorem} \label{theoremiso}
With the above notation, $F(A,g,(a_{ij})_{i  \rho  j})=
F(B,h, (b_{ij})_{i \rho  j})$ if and only if
$(A,g,(a_{ij})_{i \rho  j}) \approx (B,h, (b_{ij})_{i
\rho  j}) $. Thus $F$ induces a bijection
$$\overline{F}: \frac{U(M(\rho,k)) \times {\rm Aut}_0(\mathcal{C}) \times \mathcal{T}}{\approx} \longrightarrow{}
{\rm Iso}_{alg}({\rm End}(\mathcal{F}),{\rm End}(\mathcal{F'})).$$
\end{theorem}

\begin{proof}
Denote $F(A,g,(a_{ij})_{i  \rho  j})=\varphi $ and
$F(B,h,(b_{ij})_{i  \rho  j})=\psi$. Thus
$\varphi(E_{ij})=F_{ij}$, where $F_{ij}(w_j)=a_{ij}w_i$; here $(
\ldots, w_i, \ldots )= (\ldots,  u'_i, \ldots)A^g$. Also $\psi
(E_{ij}) = F'_{ij}$, where $F'_{ij}(w'_j)=b_{ij}w'_i$ and
$(\ldots,  w'_i, \ldots)=(\ldots, u'_i, \ldots)B^h$. Since
$\varphi=\psi$, then $F_{ij}=F'_{ij}$ for any $i  \rho
 j$. Now ${\rm Im} F_{ij}=<w_j>$ and ${\rm Im} F'_{ij}= <w'_j>$, so we must
have $w'_j=d_j w_j$ for some $d_j \in k^*$, for any $j$. Thus
$(\ldots ,  w'_i, \ldots)=(\ldots,  w_i, \ldots){\rm diag}(d_1,
\ldots,
 d_n)$, showing that $B^h=A^g \, {\rm diag}(d_1, \ldots,d_n)$.

The $(\alpha,\beta)$-block of $B^h$ may be non-zero only if
$\alpha \leq h(\beta)$, and the $(h(\beta),\beta)$-block of $B^h$
is non-zero (since the $(\beta,\beta)$-block of $B$ is non-zero).
Similarly, the $(\alpha, \beta)$-block of $A^g$ may be non-zero
only if $\alpha \leq g(\beta)$; the same holds for $A^g\, {\rm
diag}(d_1, \ ..., \ d_n)$. Since $B^h=A^g \, {\rm diag}(d_1, \
...,\ d_n)$, the $(h(\beta), \beta)$-block of $A^g$ must be
non-zero, so then $h(\beta) \leq g(\beta)$.  Similarly, since the
$(g(\beta),\beta)$-block of $A^g$ is a non-zero, we obtain that
$g(\beta) \leq h(\beta)$. Thus $g=h$.

On the other hand, $F'_{ij}(w_j)=F'_{ij}(d^{-1}_j w'_j)=d^{-1}_j
b_{ij}w'_i=d_i d^{-1}_jb_{ij}w_i$, so $F_{ij}=F'_{ij}$ requires
$a_{ij}=d_id^{-1}_j b_{ij}$ for any $i  \rho  j$.  We
conclude that $(A, g, (a_{ij})_{i  \rho  j}) \approx
(B,h, (b_{ij})_{i  \rho  j})$.

Conversely, the above computations show that the $F'_{ij}$'s
associated to each of the two triples are the same.
\end{proof}

Let us fix some notation. If $G$ and $A$ are groups, a right
action of $G$ on $A$ is a mapping $A \times G \xrightarrow{}A$,
$(a,g) \mapsto a \cdot g$, such that $(ab) \cdot g= (a \cdot g)(b
\cdot g)$ and $(a \cdot g) \cdot h=a \cdot (gh)$ for any $a,b \in
A$, $g,h \in G$. This is equivalent to giving a group morphism
$\varphi : G \xrightarrow{} {\rm Aut}(A)$, where the multiplication of
${\rm Aut}(A)$ is the opposite map composition. Indeed, one can take
$\varphi(g)(a)=a \cdot g$. In this case, the right crossed product
$G \ltimes A$ is the cartesian product $G \times A$ of sets, with
multiplication $(h \ltimes b)(g \ltimes a)= hg \ltimes (b \cdot
g)a$ for any $g,h \in G$, $a,b \in A$; here we denote the pair
$(g,a)$ by $g \ltimes a$. The group ${\rm Aut}_0(\mathcal{C})$ acts to
the right on $\mathcal{T}$ by $(a_{ij})_{i  \rho  j}
\cdot g = (a_{\tilde{g}(i)\tilde{g}(j)})_{i  \rho
j}$, thus we can form a right crossed product ${\rm Aut}_0(\mathcal{C})
\ltimes \mathcal{T}$.

Similarly, a left action of $G$ on $A$ is a mapping $G \times A
\xrightarrow{}A$, $(g,a) \mapsto g \cdot a$, such that $g \cdot
(ab)=(g \cdot a)(g \cdot b)$ and $g \cdot (h \cdot a)= (gh) \cdot
a$; this is the same with giving a group morphism $G
\xrightarrow{} {\rm Aut}(A)$, where this time the multiplication of
${\rm Aut}(A)$ is just the map composition. The left crossed product $A
\rtimes G$ is the set $A \times G$ with the multiplication $(a
\rtimes g)(b \rtimes h)=a (g \cdot b) \rtimes gh$.

Now if $G \ltimes A$ is a right crossed product, and $B$ is a
group, then a left $G \ltimes A$-action on $B$ is a pair
consisting of a left $G$-action  on $B$ and a left $A$-action on
$B$, such that
\begin{equation} \label{RC} a \cdot (g \cdot x) = g
\cdot ((a \cdot g) \cdot x) \;\; \text{ for any } a \in A, \ g \in
G, \ x \in B . \end{equation} The action of $G \ltimes A$ on $B$ is
$(g \ltimes a) \cdot x = g \cdot (a \cdot x)$. Indeed, this easily
follows from the fact that $A$, respectively $G$, embeds into $G
\ltimes A$ by $a \mapsto 1 \ltimes a$, respectively $g \mapsto g
\ltimes 1$, and $(1 \ltimes a) (g \ltimes 1)= g \ltimes (a \cdot
g)= (g \ltimes 1)(1 \ltimes (a \cdot g  ))$.

For later use (in Section \ref{sectionorbits}), we note that if $G
\ltimes A$ is a right crossed product, and $B$ is just a set, then
a right action of $G \ltimes A$ on $B$ is a pair consisting of a
right action of $G$ on $B$ and a right action of $A$ on $B$
satisfying the compatibility condition
\begin{equation}\label{comprightactionset}
(x\cdot a)\cdot g=(x\cdot g)\cdot (a\cdot g) \;\;\;\;\; \mbox{ for
any }x\in B, a\in A, g\in G.
\end{equation}

Also for later use, we mention that if $A\rtimes G$ is a left
semidirect product, and both $A$ and $G$ act to the right on a set
$B$, then $x\cdot (a\rtimes g)=(x\cdot a)\cdot g$ defines a right
action of $A\rtimes G$ on $B$, provided that
\begin{equation} \label{comprightactionset2}
(x\cdot g)\cdot a=(x\cdot (g\cdot a))\cdot g \;\;\;\;\; \mbox{ for
any }x\in B, a\in A, g\in G.
\end{equation}

We use some of these facts in the following particular situation.
The group $\mathcal{T}$ acts to the left on the group
$U(M(\rho,k))$ as follows: if $(a_{ij})_{i \rho  j} \in
\mathcal{T}$ and $A=(\alpha_{ij})_{1 \leq i,j\leq n} \in
U(M(\rho,k))$, then $(a_{ij})_{i \rho  j} \cdot A$ is the matrix
whose $(i,j)$-spot is $\alpha_{ij}a_{ij}$ if $i  \rho j$, and $0$
elsewhere.

We note that $(a_{ij})_{i  \rho  j} \cdot A \in
U(M(\rho ,k))$, its inverse being just $(a_{ij})_{i
\rho  j} \cdot A^{-1}$. The group ${\rm Aut}_0(\mathcal{C})$
also acts to the left on $U(M(\rho,k))$ by $g \cdot
A=\, ^{g^{-1}}A^{g^{-1}}$. The two actions are compatible in the
sense of (\ref{RC}), since
\begin{equation} \label{eq*}
(a_{ij})_{i  \rho  j} \cdot  \,  ^{g^{-1}}A^{g^{-1}}
=\,  ^{g^{-1}} \large( (a_{\tilde{g}(i) \tilde{g}(j)})_{i
\rho  j} \cdot A \large)^{g^{-1}} \end{equation} for
any $(a_{ij})_{i  \rho  j} \in \mathcal{T}$ and $g \in
{\rm Aut}_0(\mathcal{C})$. Indeed, if $A=(\alpha_{ij})_{1 \leq i,j \leq
n}$, it is easy to check that both sides in (\ref{eq*}) have
$a_{ij}\alpha_{\tilde{g}^{-1}(i) \tilde{g}^{-1}(j)}$ on the
$(i,j)$-spot for any $i  \rho  j$.

We obtain that ${\rm Aut}_0(\mathcal{C}) \ltimes \mathcal{T} $ acts to
the left on $U(M(\rho,k))$ by
$$(g \ltimes (a_{ij})_{i  \rho  j}) \cdot A=\, ^{g^{-1}}((a_{ij})_{i  \rho  j} \cdot A)^{g^{-1}},$$
and we can form the left crossed product $U(M(\rho ,k))
\rtimes ({\rm Aut}_0(\mathcal{C}) \ltimes \mathcal{T})$. Its
multiplication is given by \begin{equation} \label{mult} (B
\rtimes (h \ltimes (b_{ij})_{i  \rho  j})) \cdot (A
\rtimes (g \ltimes (a_{ij})_{i  \rho  j}))= (B \;\;
^{h^{-1}}((b_{ij})_{i \rho j} \cdot A)^{h^{-1}}) \rtimes (hg \ltimes
(b_{\tilde{g}(i) \tilde{g}(j)}a_{ij})_{i  \rho  j}).
\end{equation}
 Now if we apply the construction of the map $F$ to the case where $\mathcal{F'}=\mathcal{F}$, we obtain a
surjective map
$$F : U(M(\rho,k)) \rtimes ({\rm Aut}_0(\mathcal{C}) \ltimes \mathcal{T}) \xrightarrow{} {\rm Aut}({\rm End}(\mathcal{F})).$$

\begin{theorem}\label{teoremaauto}
$F$ is a morphism of groups, and it induces a group isomorphism
$$\frac{U(M(\rho,k)) \rtimes ({\rm Aut}_0(\mathcal{C}) \ltimes \mathcal{T})}{D} \simeq {\rm Aut}({\rm End}(\mathcal{F})),$$
where $D= \{ diag(d_1, \ldots,  d_n) \rtimes (Id \ltimes (d^{-1}_i
d_j)_{i  \rho  j}) \ | \ d_1,  \ldots,  d_n \in k^*
\}.$
\end{theorem}

\begin{proof}
Let $\varphi=F(A \rtimes (g \ltimes (a_{ij})_{i  \rho
j}))$, where $A=(\lambda_{ij})_{i,j}$; denote
$A^{-1}=(\overline{\lambda_{ij}})_{i,j}$. Let also $\psi=F(B
\rtimes (h \ltimes (b_{ij})_{i  \rho  j}))$, with
$B=(\mu_{ij})_{i,j}$ and $B^{-1}=(\overline{\mu}_{ij})_{i,j}$. We
will show that \begin{equation} \label{eqmorph} \psi \varphi =F( \
(B \rtimes (h \ltimes (b_{ij})_{i  \rho  j})) \cdot (A
\rtimes (g \ltimes (a_{ij})_{i  \rho  j})) \
)\end{equation} and this will prove that $F$ is a group morphism.
We know that
$$\varphi(E_{ij})=a_{ij} \displaystyle \sum_{\substack{s
\rho  \tilde{g}(i) \\ \tilde{g}(j)  \rho  t
} } \lambda_{s  \tilde{g}(i)} \overline{\lambda}_{\tilde{g}(j)
t}E_{st},$$ so then \bea \psi \varphi (E_{ij})&=&a_{ij}
\sum_{\substack{s  \rho  \tilde{g}(i) \\ \tilde{g}(j)
\rho  t }} \lambda_{s \tilde{g}(i)}
\overline{\lambda}_{\tilde{g}(j)  t} \psi (E_{st})\\
&=& a_{ij} \sum_{\substack{s  \rho  \tilde{g}(i) \\
\tilde{g}(j)  \rho  t }} \ \sum_{\substack{ p
\rho  \tilde{h}(s) \\ \tilde{h}(t)   \rho q
}}b_{st} \lambda_{s  \tilde{g}(i)}
\overline{\lambda}_{\tilde{g}(j) t } \mu_{p \tilde{h}(s)}
\overline{\mu}_{\tilde{h}(t)  q}E_{pq}\\
&=& a_{ij} \sum_{\substack{p  \rho
\tilde{h}\tilde{g}(i) \\ \tilde{h}\tilde{g}(j)  \rho q
}} \  \ \sum_{\substack{ \tilde{h}^{-1}(p)  \rho  s
\rho  \tilde{g}(i) \\ \tilde{g}(j)  \rho  t
 \rho  \tilde{h}^{-1}( q) }}b_{s\tilde{g}(i)}
b_{\tilde{g}(i)\tilde{g}(j)} b_{\tilde{g}(j)t}\lambda_{s
\tilde{g}(i)} \overline{\lambda}_{\tilde{g}(j) t } \mu_{p
\tilde{h}(s)} \overline{\mu}_{\tilde{h}(t)  q}E_{pq}\\
&=& a_{ij}b_{\tilde{g}(i)\tilde{g}(j)} \sum_{\substack{ p
\rho  \tilde{h}\tilde{g}(i) \\ \tilde{h}\tilde{g}(j)
\rho  q}}  \ \ (\sum_{\tilde{h}^{-1}(p) \rho
s  \rho  \tilde{g}(i)} \mu_{p \tilde{h}(s)} b_{s
\tilde{g}(i)} \lambda_{s \tilde{g}(i)})(\sum_{\tilde{g}(j)
\rho  t  \rho  \tilde{h}^{-1}(q)}
b_{\tilde{g}(j)t} \overline{\lambda}_{\tilde{g}(j) t}
\overline{\mu}_{\tilde{h}(t)q})E_{pq}.\eea
 But the first bracket in
the last row above is the product of the $p$-th row in $B^h$ and
the $i$-th column in $((b_{ij})_{i  \rho  j } \cdot A)^g$, thus it
is the $(p,i)$-position in the matrix $B^h((b_{ij})_{i
 \rho  j} \cdot A)^g$, or equivalently, the
$(p,\widetilde{hg}(i))$-position in
$$(B^h((b_{ij})_{i  \rho  j} \cdot A)^g)^{(hg)^{-1}}=
B^h((b_{ij})_{i  \rho  j} \cdot A)^{h^{-1}}=B \ \
^{h^{-1}}((b_{ij})_{i  \rho  j} \cdot A)^{h^{-1}}.$$

Denote $B \ \ ^{h^{-1}}((b_{ij})_{i  \rho  j} \cdot
A)^{h^{-1}}=(\nu_{ij})_{i,j}$, and the inverse of this matrix by
$(\overline{\nu}_{ij})_{i,j}$. A similar argument shows that the
second bracket above is just
$\overline{\nu}_{\tilde{h}\tilde{g}(j)q}$. Then
$$\psi \varphi (E_{ij})= a_{ij} b_{\tilde{g}(i)\tilde{g}(j)} \sum_{\substack{ p  \rho  \tilde{h}\tilde{g}(i) \\
 \tilde{h}\tilde{g}(j)  \rho  q }} \nu_{p \tilde{h}\tilde{g}(i)} \overline{\nu}_{\tilde{h}\tilde{g}(j)q}E_{pq},$$
showing  that $\psi \varphi= F(B \ \ ^{h^{-1}}((b_{ij})_{i
\rho  j} \cdot A)^{h^{-1}} \rtimes (hg \ltimes
(b_{\tilde{g}(i)\tilde{g}(j)}a_{ij})_{i  \rho  j}) )$,
which using (\ref{mult}) is just (\ref{eqmorph}).

Now $z=A \rtimes (g \ltimes (a_{ij})_{i  \rho  j})$ is
in the kernel of $F$ if and only if it is equivalent (via
$\approx$) to the identity element. In view of Theorem
\ref{theoremiso}, this is the same with $z \in D$.
\end{proof}

Now we explain how the description of the automorphism group of
$M(\rho, k)$ given in \cite{coelho} can be deduced from Theorem
\ref{teoremaauto}. We recall a few basic things about semidirect
products. We first note that if a group $G$ acts to the right on a
group $A$, with action denoted by $a\cdot g$ for $a\in A$ and
$g\in G$, then there is also a left action of $G$ on $A$, defined
by $g\cdot a=a\cdot g^{-1}$, and the associated right and left
semidirect products are isomorphic. Indeed, $\phi:A \rtimes
G\rightarrow G\ltimes A$, $\phi(a\rtimes g)=g\ltimes (g^{-1}\cdot
a)$, is a group isomorphism, with inverse $\phi^{-1}(g\ltimes
a)=a\cdot g^{-1}\rtimes g$. Secondly, if we have a double
semidirect product $B\rtimes (A\rtimes G)$, then the left action
of $A\rtimes G$ on $B$ induces (via the usual embeddings of $A$
and $G$) left actions of $A$ and $G$ on $B$. Moreover, $G$ acts to
the left on the group $B\rtimes A$ by $g\cdot (x\rtimes a)=g\cdot
x\rtimes g\cdot a$ for any $g\in G, x\in B, a\in A$, and the map
$\gamma: B\rtimes (A\rtimes G)\ra (B\rtimes A)\rtimes G$,
$\gamma(x\rtimes (a\rtimes g))=(x\rtimes a)\rtimes g$, is an
isomorphism of groups. Finally, if $A\rtimes G$ is a left
semidirect product, and $N$ is a normal subgroup of $A$ which is
invariant to the action of $G$, then $N\rtimes 1$ is a normal
subgroup of $A\rtimes G$ and $\frac{A\rtimes G}{N\rtimes G}\simeq
\frac{A}{N}\rtimes G$, where the action of $G$ on $\frac{A}{N}$ is
the one induced by the action of $G$ on $A$.

Using these remarks, we see that \bea U(M(\rho,k)) \rtimes ({\rm
Aut}_0(\mathcal{C}) \ltimes \mathcal{T})&\simeq &U(M(\rho,k))
\rtimes (\mathcal{T} \rtimes {\rm
Aut}_0(\mathcal{C}))\\
&\simeq& (U(M(\rho,k)) \rtimes \mathcal{T}) \rtimes {\rm
Aut}_0(\mathcal{C})\eea and in the third double semidirect product
the action of ${\rm Aut}_0(\mathcal{C})$ on $U(M(\rho,k)) \rtimes
\mathcal{T}$ is given by
$$g\cdot (A\rtimes (a_{ij})_{i  \rho  j})= \, ^{g^{-1}}A^{g^{-1}}
\rtimes (a_{ij})_{i  \rho  j}\cdot g^{-1}.$$ The image of $D$
through these isomorphisms in $(U(M(\rho,k)) \rtimes \mathcal{T})
\rtimes {\rm Aut}_0(\mathcal{C})$ is $D_0\rtimes 1$, where
$$D_0=\{ {\rm diag}(d_1, \ldots,  d_n) \rtimes  (d^{-1}_i
d_j)_{i  \rho  j} \ | \ d_1,  \ldots,  d_n \in k^* \}.$$ Clearly
$D_0$ is a normal subgroup in $U(M(\rho,k)) \rtimes \mathcal{T}$,
since $D$ is a normal subgroup of $U(M(\rho,k)) \rtimes ({\rm
Aut}_0(\mathcal{C}) \ltimes \mathcal{T})$, and it is easy to check
that $D_0$ is invariant under the action of ${\rm
Aut}_0(\mathcal{C})$. Thus we obtain that \bea \frac{U(M(\rho,k))
\rtimes ({\rm Aut}_0(\mathcal{C}) \ltimes \mathcal{T})}{D}&\simeq&
\frac{(U(M(\rho,k)) \rtimes \mathcal{T}) \rtimes {\rm
Aut}_0(\mathcal{C})}{D_0\rtimes 1}\\
&\simeq& \frac{U(M(\rho,k)) \rtimes \mathcal{T}}{D_0}\rtimes {\rm
Aut}_0(\mathcal{C}).\eea

Following \cite{coelho}, we consider an undirected graph $\Delta$,
whose vertices are all elements $i\in \{ 1,\ldots,n\}$ such that
the equivalence class $\hat{i}$ is not an isolated point in the
poset $(\mathcal{C},\leq )$; we say that $\alpha \in \mathcal{C}$
is an isolated point if any of $\alpha\leq \beta$ and $\beta\leq
\alpha$ implies that $\beta=\alpha$.  If $i$ and $j$ are vertices
of $\Delta$, there is an edge connecting $i$ and $j$ if and only
if $\hat{i}$ and $\hat{j}$ are not equal, but they are in relation
$\leq$ (in any possible way). If $\Delta_1,\ldots ,\Delta_z$ are
the connected components of $\Delta$, choose a tree $T_\ell$ in
each $\Delta_\ell$ (note that there are several possible such
choices). Now let $\mathcal{G}$ be the subgroup of $\mathcal{T}$
consisting of all $(a_{ij})_{i  \rho j}$'s for which $a_{ij}=1$
whenever $i\rho j$ and $i$, $j$ are vertices of $\Delta$, joined
by an edge of some $T_\ell$, or when both $i$ and
$j$ lie in an isolated equivalence class in $\mathcal{C}$.\\
Also, let $\mathcal{I}$ be the group of inner automorphisms of
$M(\rho, k)$, i.e. $\mathcal{I}=\{ C_A\, |\, A\in U(M(\rho
,k))\}$, where $C_A(X)=AXA^{-1}$ for any $X\in M(\rho ,k)$. Then
there is a left action of $\mathcal{G}$ on $\mathcal{I}$ given by
$(a_{ij})_{i  \rho j}\cdot C_A=C_{(a_{ij})_{i  \rho j}\cdot A}$.
We prove that
\begin{equation}\label{eqisofactor}
\frac{U(M(\rho,k)) \rtimes \mathcal{T}}{D_0}\simeq
\mathcal{I}\rtimes \mathcal{G}
\end{equation}
and this will show that
$$Aut(M(\rho,k))\simeq \frac{U(M(\rho,k)) \rtimes ({\rm
Aut}_0(\mathcal{C}) \ltimes \mathcal{T})}{D} \simeq
(\mathcal{I}\rtimes \mathcal{G})\rtimes {\rm
Aut}_0(\mathcal{C}),$$ thus recovering the description of the
automorphism group of the structural matrix algebra given in
\cite{coelho}.

For proving (\ref{eqisofactor}), define
$$\Psi:\mathcal{I}\rtimes \mathcal{G}\ra \frac{U(M(\rho,k)) \rtimes
\mathcal{T}}{D_0},\; \Psi (C_A\rtimes (a_{ij})_{i  \rho
j})=\overline{A\rtimes (a_{ij})_{i  \rho j}},$$ where
$\overline{y}$ denotes the class of $y$ modulo $D_0$. We note that
$\Psi$ is well defined since $C_A=C_B$ if and only if $A^{-1}B$
lies in the center of $M(\rho, k)$, and this is the set of
diagonal matrices constant on all $i$th diagonal positions with
$\hat{i}$ in the same connected component of $\mathcal{C}$. Then
$B=A{\rm diag}(d_1,\ldots,d_n)$ for such a central diagonal matrix
${\rm diag}(d_1,\ldots,d_n)$. Since $d_i^{-1}d_j=1$ for any $i\rho
j$, we have that ${\rm diag}(d_1,\ldots,d_n)\rtimes (1)_{i\rho
j}\in D_0$. But $B\rtimes (a_{ij})_{i  \rho j}=({\rm
diag}(d_1,\ldots,d_n)\rtimes (1)_{i\rho j})(A\rtimes (a_{ij})_{i
\rho j})$, so $\overline{B\rtimes (a_{ij})_{i  \rho
j}}=\overline{A\rtimes (a_{ij})_{i  \rho j}}$.

We show that $\Psi$ is injective. Indeed, if $\Psi (C_A\rtimes
(a_{ij})_{i  \rho j})$ is trivial, then $A={\rm diag}(d_1,\ldots
,d_n)$ and $a_{ij}=d_i^{-1}d_j$ for any $i\rho j$, where
$d_1,\ldots,d_n$ are some non-zero scalars. Since $a_{ij}=1$ for
any $i,j$ joined by an edge of some $T_\ell$, and $T_\ell$ is a
tree, we see that $d_i$ must be the same when $i$ runs through the
vertices of a fixed $\Delta_\ell$. Also, $a_{ij}=1$ for $i,j$ in
the same isolated equivalence class $\alpha\in \mathcal{C}$, so
$d_i$ is constant for $i$ in such a class. We conclude that $A$ is
a central element, so $C_A$ is the identity, and all $a_{ij}$'s
with $i\rho j$ are equal to 1.

To prove that $\Psi$ is surjective, it is enough to show that for
any $A\rtimes (a_{ij})_{i  \rho j} \in U(M(\rho,k)) \rtimes
\mathcal{T}$ there exists $B\rtimes (b_{ij})_{i  \rho j}\in
\mathcal{I}\rtimes \mathcal{G}$ with $\overline{A\rtimes
(a_{ij})_{i  \rho j}}=\overline{B\rtimes (b_{ij})_{i  \rho j}}$.
We first show that there are $d_1,\ldots ,d_n\in k^*$ such that
$a_{ij}=d_id_j^{-1}$ for any $i,j$ joined by an edge of some
$T_\ell$, and also for any $i,j$ in the same isolated equivalence
class. Indeed, for the tree $T_{\ell}$ we can fix some vertex
$i_0$, set $d_{i_0}=1$, and then define $d_i$ for any other vertex
in $T_\ell$ by induction on the distance from $i_0$ to $i$ in the
tree $T_\ell$, using at each step the desired condition
$a_{ij}=d_id_j^{-1}$. For an isolated equivalence class, say $\{
i_1,\ldots ,i_p\}$, set $d_{i_p}=1$, and $a_{i_1,i_p}=d_{i_1},
\ldots ,a_{i_{p-1},i_p}=d_{i_{p-1}}$. Then
$a_{i_r,i_s}=a_{i_r,i_p}a_{i_p,i_s}=a_{i_r,i_p}a_{i_s,i_p}^{-1}=d_{i_r}d_{i_s}^{-1}$.
Now $$({\rm diag}(d_1, \ldots,  d_n) \rtimes  (d^{-1}_i d_j)_{i
\rho  j})(A\rtimes (a_{ij})_{i  \rho j})={\rm diag}(d_1, \ldots,
d_n)((d^{-1}_i d_j)_{i \rho  j}\cdot A)\rtimes
(d_i^{-1}d_ja_{ij})_{i\rho j}$$ and $(d_i^{-1}d_ja_{ij})_{i\rho
j}\in \mathcal{G}$, so we are done.

\section{Graded flags and associated gradings of
matrices}\label{sectiongradedflags}

Let $G$ be a group. A $G$-graded vector space is a vector space
$V$ with a decomposition $V=\displaystyle \bigoplus_{g \in G}V_g$,
where each $V_g$ is a subspace. The elements of $\displaystyle
\bigcup_{g \in G}V_g$ are called the homogeneous elements of $V$.
Each $v \in V$ is uniquely written as $v=\displaystyle \sum_{g \in
G}v_g$, $v_g \in V_g$.

\begin{definition}
A $G$-graded $\rho$-flag is a $\rho$-flag $(V,(V_{\alpha})_{\alpha
\in \mathcal{C}})$ such that $V$ is a $G$-graded vector space, and
the basis $B$ from Definition \ref{de} consists of homogeneous
elements.\\
If $\mathcal{F}=(V,(V_{\alpha})_{\alpha \in \mathcal{C}})$ and
$\mathcal{F}'=(V',(V'_{\alpha})_{\alpha \in \mathcal{C}})$ are
$G$-graded $\rho$-flags, then a morphism of graded flags from
$\mathcal{F}$ to $\mathcal{F}'$ is a morphism $f:V\ra V'$ of
$\rho$-flags, which is also a morphism of graded vector spaces.
\end{definition}

Note that in a graded flag any $V_{\alpha}$ is a graded vector
subspace, as it has a basis of homogeneous elements.

If $\mathcal{F}=(V, (V_{\alpha})_{\alpha \in \mathcal{C}})$ is a
$G$-graded $\rho$-flag and $\sigma \in G$, define
$$ {\rm End}(\mathcal{F})_{\sigma}=\{ f \in {\rm End}(\mathcal{F}) \ | \ f(V_g) \subseteq V_{\sigma g} \ \text{for any } g \in G \}.$$

\begin{proposition} \label{defEND}
${\rm End}(\mathcal{F})=\displaystyle \bigoplus_{\sigma \in
G}{\rm End}(\mathcal{F})_{\sigma}$, and this decomposition makes
${\rm End}(\mathcal{F})$  a $G$-graded algebra.
\end{proposition}
\begin{proof}
It is clear that ${\rm End}(\mathcal{F})_{\sigma}$ can be non-zero
only for $\sigma \in ({\rm supp} \ V) \cdot ({\rm supp} \
V)^{-1}$, where ${\rm supp} \ V= \{ g \in G \ | \ V_g \neq 0 \}$.
Thus only finitely many ${\rm End}(\mathcal{F})_{\sigma}$ are
nonzero.

In order to see that $\displaystyle \sum_{\sigma \in G} {\rm
End}(\mathcal{F})_{\sigma}$ is a direct sum, choose some
$f^{\sigma} \in {\rm End}(\mathcal{F})_{\sigma}$ for each $\sigma
\in G$. If $\displaystyle \sum_{\sigma \in G}f^{\sigma}=0$, then
$0= ( \displaystyle \sum_{\sigma \in G}f^{\sigma} ) ( V_g
)=\displaystyle \sum_{\sigma \in G}f^{\sigma}(V_{g})$ for any $g
\in G$. Since $f^{\sigma}(V_g) \subset V_{\sigma g}$, this shows
that $f^{\sigma}(V_g)=0$ for any $\sigma$ and any $g$, and we get
that $f^{\sigma}=0$ for any $\sigma$.

Now let $f \in {\rm End}(\mathcal{F})$. For any $\sigma \in G$ define
the linear maps $f_{\sigma}:V \xrightarrow{} V$ such that
$(f_{\sigma})(v_g)=f(v)_{\sigma g}$ for any $g \in G$, $v_g \in
V_g$. Then
$$f_{\sigma}(V_{\alpha})=\sum_{g \in G}f_{\sigma}((V_{\alpha})_g) \subset \sum_{g \in G}f(V_{\alpha})_{\sigma g} \subset V_{\alpha}$$
for any $\alpha \in \mathcal{C}$, so $f_{\sigma} \in
{\rm End}(\mathcal{F})_{\sigma}$. Moreover, it is clear that
$f=\displaystyle \sum_{\sigma \in G}f_{\sigma}$, so
${\rm End}(\mathcal{F})= \displaystyle \bigoplus_{\sigma \in
G}{\rm End}(\mathcal{F})_{\sigma}$.

Obviously, ${\rm End}(\mathcal{F})_{\sigma}{\rm End}(\mathcal{F})_{\tau}
\subset {\rm End}(\mathcal{F})_{\sigma \tau}$ for any $\sigma , \tau \in
G$, so ${\rm End}(\mathcal{F})$ is a $G$-graded algebra.
\end{proof}

We will denote by ${\rm END}(\mathcal{F})$ the algebra ${\rm
{\rm End}}(\mathcal{F})$, regarded with the $G$-grading defined in
Proposition \ref{defEND}. This grading transfers via the
isomorphism defined in the proof of Proposition \ref{isoEndMat} to
a $G$-grading on the structural matrix algebra $M(\rho,k)$. If
$g_1,\ldots, g_n$ are the degrees of the basis elements
$v_1,\ldots,v_n$, then each matrix unit $e_{ij}$ with $i\rho j$ is
a homogeneous element of degree $g_ig_j^{-1}$ in this grading.

\begin{definition}
A $G$-grading on the algebra $M(\rho,k)$ is called a good grading
if $e_{ij}$ is a homogeneous element for any $i,j$ with $i \rho
j$.
\end{definition}

Gradings on $M(\rho,k)$ arising from graded flags are good
gradings. We note that giving a good $G$-grading on $M(\rho,k)$ is
equivalent to  giving a family $(u_{ij})_{i \rho  j}$ of elements
of $G$ such that $u_{ij}u_{jr}=u_{ir}$ for any $i,j,r$ with $i
\rho  j$ and $ j  \rho  r$. If we regard this family as a function
$u:\rho \ra G$, defined by $u(i,j)=u_{ij}$ for any $i,j$ with
$i\rho j$, then $u$ is just a transitive function on $\rho$ with
values in $G$, in the terminology of \cite{nowicki}.

Examples of a transitive functions on $\rho$ can be obtained as
follows. Let $g_1,\ldots,g_n\in G$, and let $u_{ij}=g^{}_i
g^{-1}_j$ for any $i,j$ with $i  \rho j$. Then $(u_{ij})_{i \rho
j}$ is a transitive function on $\rho$. A transitive function on
$\rho$ is called trivial if it is obtained in this way.
 Clearly, the good $G$-gradings corresponding to trivial transitive
functions on $\rho$ are precisely the gradings obtained from
graded flags as above. It is an interesting question whether all
good gradings arise from graded flags, or equivalently\\

{\bf Question.} {\it Let $\rho$ be a preorder relation. Is it true
that for any group $G$ all transitive functions $u:\rho\ra G$ are
trivial?}\\

Several variations of the above questions can be formulated, for
example to determine for a fixed $\rho$ all groups $G$ such that
any transitive function on $\rho$ with values in $G$ is trivial.

The following shows that the problem posed in the question above
reduces to answering it for the associated poset
$(\mathcal{C},\leq )$. We also give an equivalent formulation
involving the associated graph $\Gamma$. The equivalence between
(1) and (2) in the next Proposition was proved in \cite{nowicki}.

\begin{proposition} \label{prop*}
Let $G$ be a group. The following are equivalent:\\
{\rm (1)} Any transitive function $u:\rho \ra G$ is trivial.\\
{\rm (2)} Any transitive function $w:\,\leq \, \ra G$ is trivial,
where
$\leq$ is the partial order on $\mathcal{C}$. \\
{\rm (3)} For any function $v : \Gamma_1 \ra G$ such that $v(a_1)
\ldots v(a_r)= v(b_1) \ldots v(b_s) $ for any paths $a_1 \ldots
a_r$ and $b_1 \ldots b_s$ in $\Gamma$ with $s(a_1)=s(b_1)$ and
$t(a_r)=t(b_s)$, there exists a function $f : \Gamma_0 \ra G$ such
that $v(a)=f(s(a))f(t(a))^{-1}$ for any $a \in \Gamma_1$.
 \end{proposition}

\begin{proof} Denote $\mathcal{C}=\{ \alpha_1 , ...,
\alpha_h \}$ and pick $i_1 \in \alpha_1, ..., i_h \in \alpha_h$.\\
(1)$\Rightarrow$(2) Let $w:\,\leq\, \ra G$ be a transitive
function. Then $u:\rho \ra G$ defined by
$u(i,j)=w(\hat{i},\hat{j})$ for any $i,j$ with $i\rho j$, is a
transitive function, so there exists a function $f:\{ 1,\ldots
,n\}\ra G$ such that $u(i,j)=f(i)f(j)^{-1}$ for any $i,j$ with
$i\rho j$. Now define $\overline{f}:\mathcal{C}\ra G$,
$\overline{f}(\alpha_q)=f(i_q)$ for any $1\leq q\leq h$. Then if
$\alpha_q\leq \alpha_p$, we have
$w(\alpha_q,\alpha_p)=u(i_q,i_p)=f(i_q)f(i_p)^{-1}=\overline{f}(\alpha_q)\overline{f}(\alpha_p)^{-1}$,
so $w$ is trivial.\\
(2)$\Rightarrow$(3) Let $v : \Gamma_1 \ra G$ be as in (3). Define
$w:\,\leq\, \ra G$ as follows. If $\alpha<\beta$, let $a_1\ldots
a_r$ be a path in $\Gamma$ starting at $\alpha$ and ending at
$\beta$; we define $w(\alpha,\beta)=v(a_1)\ldots v(a_r)$, and we
note that the definition does not depend on the path, taking into
account the property satisfied by $v$. We also define
$w(\alpha,\alpha)=e$, the neutral element of $G$, for any $\alpha
\in \mathcal{C}$. Then $w$ is a transitive function, so there
exists $f : \mathcal{C} \ra G$ such that
$w(\alpha,\beta)=f(\alpha)f(\beta)^{-1}$ for any $\alpha\leq
\beta$. In particular $v(a)=f(s(a))f(t(a))^{-1}$ for any $a \in
\Gamma_1$.\\
(3)$\Rightarrow$(1) Let $u:\rho \ra G$ be a transitive function.
 Define $v : \Gamma_1 \xrightarrow{} G$ as follows: if $a \in
\Gamma_1$ with $s(a)=\alpha_p$, $t(a)=\alpha_q$, then $v(a)=u(i_p,
i_q)$. If $a_1\ldots a_r$ and $b_1\ldots b_s$ are paths in
$\Gamma$ with $s(a_1)=s(b_1)=\alpha_p$ and
$t(a_r)=t(b_s)=\alpha_q$, it is clear that $v(a_1)\ldots v(a_r)=
v(b_1)\ldots v(b_s)=u(i_p, i_q)$. Then there exist $z_1, ..., z_n
\in G$ such that $v(a)=z_p z_q^{-1}$ for any arrow $a$ with
$s(a)=\alpha_p$, $t(a)=\alpha_q$, $1 \leq p,q \leq h$.

Now define the family $(g_i )_{1 \leq i \leq n }$ of elements of
$G$ by $g_i = u(i, i_p)z_p$ for any $i$, where $p$ is such that $i
\in \alpha_p$ (note that $u(i, i_p)$ makes sense since $i, i_p \in
\alpha_p$). Then if $i$ and $j$ are such that $i  \rho   j$, let
$i \in \alpha_p$, $j \in \alpha_q$. We know that $g_i=u(i,
i_p)z_p$, $g_j=u(j, i_q)z_q$ and $u(i_p, i_q)=z_pz_q^{-1}$. Then
\bea u(i,j)&=&u(i, i_p)u(i_p, i_q) u(i_q, j)\\
&=&g_i z_p^{-1} z_p z_q^{-1} u(j, i_q)^{-1}\\
&=&g_i z_q^{-1} z_q g_j^{-1}\\
&=& g_i g_j^{-1}. \eea
\end{proof}

The next result gives an easy way to check whether a function $v$
as in Proposition \ref{prop*} (3) arises from a function
$f:\Gamma_0\ra G$, by looking at the cycles of the undirected
graph associated with $\Gamma$. We consider the graph
$\tilde{\Gamma}$, constructed from $\Gamma$ by 'doubling the
arrows'. It has the same  vertices as $\Gamma$ thus
$\tilde{\Gamma}_0=\Gamma_0 $. For any arrow $a \in \Gamma_1$ from
$\alpha$ to $\beta$, we consider an arrow $\tilde{a}$ from $\beta$
to $\alpha$, and define $\tilde{\Gamma}_1=\Gamma_1 \cup \{
\tilde{a} \ | \ a \in \Gamma_1 \}$. If $G$ is a group and $v :
\Gamma_1 \ra G$ is a function, we denote by
$\tilde{v}:\tilde{\Gamma}_1 \ra G$ the function whose restriction
to $\Gamma_1$ is $v$, and such that
$\tilde{v}(\tilde{a})=v(a)^{-1}$ for any $a \in \Gamma_1$.

\begin{proposition} \label{proptest}
Let $G$ be a group, and let $v : \Gamma_1 \xrightarrow{}G$ such
that $v(a_1) \ldots v(a_r)= v(b_1) \ldots v(b_s)$ for any paths
$a_1 \ldots a_r$ and $b_1 \ldots b_s$ in $\Gamma$ with
$s(a_1)=s(b_1)$ and $t(a_r)=t(b_s)$. Then the following are
equivalent.\\
{\rm (1)} There exists a function $f : \Gamma_0 \ra G$ such that
$v(a)=f(s(a))f(t(a))^{-1}$ for any $a \in \Gamma_1$.\\
{\rm (2)} For any cycle $z_1  ...  z_m$ in $\tilde{\Gamma}$, with
$z_1, ..., z_m \in \tilde{\Gamma}_1$ (this corresponds to a cycle
in the undirected graph $\Gamma^u$ associated with $\Gamma$) one
has $\tilde{v}(z_1)
 \ldots \tilde{v}(z_m)=1$.
\end{proposition}

\begin{proof}
(1)$\Rightarrow$(2) For any $a\in \Gamma_1$ starting from $\alpha$
and ending at $\beta$, we have
$\tilde{v}(a)=v(a)=f(\alpha)f(\beta)^{-1}$ and
$\tilde{v}(\tilde{a})=f(\beta)f(\alpha)^{-1}$. Thus
$\tilde{v}(z)=f(s(z))f(t(z))^{-1}$ for any $z \in
\tilde{\Gamma}_1$. Now it is clear that $\tilde{v}(z_1) \ ... \
\tilde{v}(z_m)=1$ for any cycle $z_1 \ ... \ z_m$ in
$\tilde{\Gamma}$.\\
(2)$\Rightarrow$(1)
 We construct a function $f : \Gamma_0
\xrightarrow{} G$ satisfying the desired property.  It is clear
that we can reduce to the case where $\Gamma$ is connected (and
putting together the functions constructed for the connected
components). Choose some $\omega \in \Gamma_0$ and take $f(\omega)
$ be an arbitrary element of $G$. If $\alpha \in \Gamma_0$, let
$z_1\ldots z_{d}$ be a path from $\alpha$ to $\omega$ in
$\tilde{\Gamma}$. Define $f(\alpha)=\tilde{v}(z_1)\ldots
\tilde{v}(z_{d})f(\omega)$; this does not depend on the path $z_1
\ldots z_{d}$, since for another path $y_1 \ldots y_h$ from
$\alpha$ to $\omega$, $z_1 \ldots z_d y_h^{-1} \ldots y_1^{-1}$ is
a cycle in $\tilde{\Gamma}$, where $y^{-1}$ denotes the arrow
opposite to $y$ (thus $a^{-1}=\tilde{a}$, and $\tilde{a}^{-1}=a$
for any $a\in \Gamma_1$). It is clear that
$\tilde{v}(y^{-1})=\tilde{v}(y)^{-1}$ for any $y \in
\tilde{\Gamma}_1$. Then $\tilde{v}(z_1) \ldots \tilde{v}(z_d)
\tilde{v}(y_{h}^{-1}) \ldots \tilde{v}(y_1^{-1})=1$, so
$\tilde{v}(z_1) \ldots \tilde{v}(z_d)= \tilde{v}(y_1) \ldots
\tilde{v}(y_h)$. Now if $a \in \Gamma_1$ is an arrow from $\alpha$
to $\beta$, let $z_1 \ ... \ z_d$ be a path from $\beta$ to
$\omega$ in $\tilde{\Gamma}$. Then $f(\beta)=\tilde{v}(z_1)  \ ...
\ \tilde{v}(z_d)f(\omega)$, and
$f(\alpha)=\tilde{v}(a)\tilde{v}(z_1)\ ... \
\tilde{v}(z_d)f(\omega)=\tilde{v}(a)f(\beta)=v(a)f(\beta)$, so
$v(a)=f(\alpha)f(\beta)^{-1}$.
\end{proof}

Let $F(\Gamma)$ be the free group generated by the set $\Gamma_1$
of arrows of $\Gamma$. Let $A(\Gamma)$ be the subgroup of
$F(\Gamma)$ generated by all elements of the form $a_1\ldots
a_rb_p^{-1}\ldots b_1^{-1}$, where $a_1\ldots a_r$ and $b_1\ldots
b_p$ are two paths (in $\Gamma$) starting from the same vertex and
terminating at the same vertex. We also consider the subgroup
$B(\Gamma)$ of $F(\Gamma)$ generated by all elements of the form
$a_1a_2^{\varepsilon_2}\ldots a_m^{\varepsilon_m}$, where
$a_1,\ldots ,a_m$ are arrows forming in this order a cycle in the
undirected graph obtained from $\Gamma$ when omitting the
direction of arrows, and $\varepsilon_i=1$ if $a_i$ is in the
direction of the directed cycle given by $a_1$, and
$\varepsilon_i=-1$ otherwise. Clearly $A(\Gamma)\subseteq
B(\Gamma)$, since any generator of $A(\Gamma)$ lies in
$B(\Gamma)$. Now we can give an answer to the question posed above
in terms of these groups. We recall that for a group $X$ and a
subgroup $Y$ of $X$, the normal closure $Y^N$ of $Y$ is the
smallest normal subgroup of $X$ containing $Y$. The elements of
$Y^N$ are all products of conjugates of elements of $Y$.
\begin{proposition}
With notation as above, the following are equivalent.\\
$(1)$ For any group $G$, any transitive function $u:\rho \ra G$ is
trivial.\\
$(2)$ $A(\Gamma)^N=B(\Gamma)^N$.\\
$(3)$ Any generator $b$ of $B(\Gamma)$ can be written in the form
$b=g_1x_1g_1^{-1}\ldots g_mx_mg_m^{-1}$ for some positive integer
$m$, some $g_1,\ldots,g_m\in F(\Gamma)$ and some $x_1,\ldots,x_m$
among the generators in the construction of $A(\Gamma)$.
\end{proposition}
\begin{proof}
By Propositions \ref{prop*} and \ref{proptest}, (1) is equivalent
to the fact that for any group $G$ and any group morphism
$f:F(\Gamma)\ra G$ such that $f(A(\Gamma))=1$, we also have
$f(B(\Gamma))=1$. Indeed, giving a function $v:\Gamma_1\ra G$ is
the same with giving a group morphism $f:F(\Gamma)\ra G$;
moreover, $v$ satisfies $v(a_1) \ldots v(a_r)= v(b_1) \ldots
v(b_s) $ for any paths $a_1 \ldots a_r$ and $b_1 \ldots b_s$ in
$\Gamma$ with $s(a_1)=s(b_1)$ and $t(a_r)=t(b_s)$, if and only if
$f(A(\Gamma))=1$. On the other hand, such a $v$ is trivial if and
only if $f(B(\Gamma))=1$.

Now if (1) holds, then the projection $F(\Gamma)\ra
F(\Gamma)/A(\Gamma)^N$ is trivial on $A(\Gamma)$, so it must be
trivial on $B(\Gamma)$. This shows that $B(\Gamma)\subseteq
A(\Gamma)^N$, or $A(\Gamma)^N=B(\Gamma)^N$. Conversely, if (2)
holds, then any group morphism $f:F(\Gamma)\ra G$ which is trivial
on $A(\Gamma)$ is also trivial on $A(\Gamma)^N$, and then also on
$B(\Gamma)$, so (1) holds.

The equivalence between (2) and (3) follows from the description
of the normal closure of a subgroup, and the fact that the given
set of generators of $A(\Gamma)$ is closed under inverse.
\end{proof}

\begin{example} \label{exemplu1}
Assume that $\rho$ is a preorder relation such that the associated
graph $\Gamma$ is of the form
$$
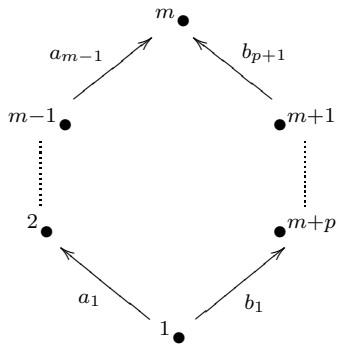
\xymatrix{&^m\bullet&\\^{m-1}\bullet\ar[ur]^{a_{m-1}}&&\bullet^{m+1}\ar[ul]_{b_{p+1}}\\
^2\bullet\ar@{.}[u]&&
\bullet^{m+p}\ar@{.}[u]\\&^1\bullet\ar[ul]^{a_1}\ar[ur]_{b_1}&}$$
for some integers $m\geq 3$ and $p\geq 1$. Thus there are two
paths from 1 to $m$, these are $a_1\ldots a_{m-1}$ and $b_1\ldots
b_{p+1}$. Then for any group $G$, any transitive function
$u:\rho\ra G$ is trivial. Indeed, if $v$ is a function as in
Proposition \ref{proptest}, then $v(a_1)\ldots
v(a_{m-1})=v(b_1)\ldots v(b_{p+1})$, and then the condition in (2)
of the mentioned Proposition is obviously satisfied, as it is
clear that the associated undirected graph has just one cycle.

Alternatively, $A(\Gamma)$ is the cyclic group generated by
$x=a_1\ldots a_{m-1}b_{p+1}^{-1}\ldots b_1^{-1}$, while the
generators of $B(\Gamma)$ are all conjugates of $x$ or $x^{-1}$.
For example, $a_2\ldots a_{m-1}b_{p+1}^{-1}\ldots
b_1^{-1}a_1=a_1^{-1}xa_1$. Thus $A(\Gamma)^N=B(\Gamma)^N$.
\end{example}

\begin{example} \label{exemplu2}
Assume that $\rho$ is a preorder relation such that the associated
graph $\Gamma$ is of the form
$$
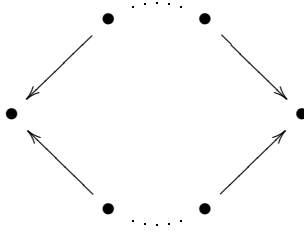
\xymatrix{  & \bullet \ar[dl] \ar@{.}@/^/[r] & \bullet \ar[dr]& \\ \bullet  & &  & \bullet  \\ & \bullet \ar[ul] \ar@{.}@/_/[r]& \bullet \ar[ur] & }
$$
Thus the  un-directed graph $\Gamma^u$ associated to $\Gamma$ is
cyclic, and in $\Gamma$ there are at least two vertices where both
adjacent arrows terminate (equivalently, $\Gamma^u$ is cyclic and
$\Gamma$ is not of the type in Example \ref{exemplu1}). Then for
any non-trivial group $G$, there exist transitive functions
$u:\rho \ra G$ that are not trivial. Indeed, it is enough to show
that the condition (3) in Proposition \ref{prop*} does not hold.
Since $\Gamma$ does not have two different paths starting from the
same vertex and ending at the same vertex, we only have to see
that not any function $v:\Gamma_1\ra G$ is given by
$v(a)=f(s(a))f(t(a))^{-1}$, $a\in \Gamma_1$, for some function
$f:\Gamma_0\ra G$. This is clear, for instance we can take $v$ to
be the identity of $G$ on all but one arrows.

We note that $A(\Gamma)$ is trivial, since there are no two
different paths starting from the same vertex  and terminating at
the same vertex. On the other hand, $B(\Gamma)$ is not trivial,
since $\Gamma^u$ has a cycle.

The simplest example of such a graph is
$$
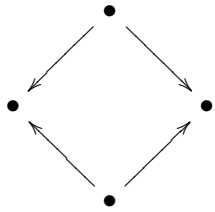
\xymatrix{ & \bullet \ar[dl] \ar[dr] & \\ \bullet & & \bullet \\ & \bullet \ar[ul] \ar[ur] & }$$
and the corresponding structural matrix algebra, whose not all
good gradings arise from graded flags, is
$$\left(
\begin{array}{cccc}
k&0&k&k\\
0&k&k&k\\
0&0&k&0\\
0&0&0&k
\end{array}
\right).$$
\end{example}

\begin{example}
If we construct a graph $\Delta$ by taking a graph $\Gamma$ as in
Example \ref{exemplu2}, adding a vertex $v$, and adding an arrow
from each vertex in $\Gamma$ where both adjacent arrows terminate
to $v$, as in the picture below

$$
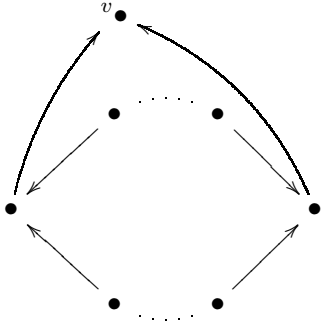
\xymatrix{ & ^v\bullet & & \\ & \bullet \ar[dl] \ar@{.}@/^/[r] & \bullet \ar[dr]& \\ \bullet \ar@/^/[uur] & &  & \bullet \ar@/_1pc/[uull] \\ & \bullet \ar[ul] \ar@{.}@/_/[r]& \bullet \ar[ur] & }$$
then all transitive functions (on the corresponding preordered
set) are trivial. For simplicity, we explain this in the case
where $\Delta$ is

$$ 
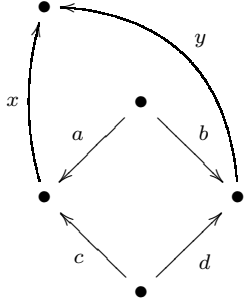
\xymatrix{ \bullet & & \\ & \bullet \ar[dl]_a \ar[dr]^b & \\ \bullet \ar@/^/[uu]^x& & \bullet \ar@/_2pc/[uull]_y \\ & \bullet \ar[ul]^c \ar[ur]_d &  }$$
but the argument is the same in general. In this case, $A(\Delta)$
is generated by $g=axy^{-1}b^{-1}$ and $h=cxy^{-1}d^{-1}$, while
$B(\Delta)$ is generated by certain conjugates of $g$ and $h$ (as
explained in Example \ref{exemplu2}), $ac^{-1}db^{-1}$, and
certain conjugates of the latter. But
$ac^{-1}db^{-1}=g(db^{-1})^{-1}h^{-1}(db^{-1})$, showing that
$A(\Delta)^N=B(\Delta)^N$.
\end{example}

\section{Isomorphisms between graded endomorphism algebras}
\label{sectionisograded}

If $V=\oplus_{g\in G}V_g$ is a $G$-graded vector space, then for
any $\sigma\in G$ the right $\sigma$-suspension of $V$ is the
$G$-graded vector space $V(\sigma)$ which is just $V$ as a vector
space, with the grading shifted by $\sigma$ as follows:
$V(\sigma)_g=V_{g\sigma}$ for any $g\in G$. Also, the left
$\sigma$-suspension of $V$, denoted by $(\sigma)V$, is the vector
space $V$ with the grading given by $((\sigma)V)_g=V_{\sigma g}$.
For any $\sigma,\tau\in G$ one has
$(V(\sigma))(\tau)=V(\tau\sigma)$,
$(\sigma)((\tau)V)=(\tau\sigma)V$ and
$((\sigma)V)(\tau)=(\sigma)(V(\tau))$. The fact that there are two
types of suspensions for a graded vector space $V$ can be
explained by the fact that $V$ is an object in the category of
left graded modules, and also an object in the category of right
graded modules over the algebra $k$, regarded with the trivial
$G$-grading. Then $V(\sigma)$ and $(\sigma)V$ are just the
suspensions when $V$ is regarded in these categories.

If $V$ and $W$ are $G$-graded vector spaces and $\sigma\in G$, we
say that a linear map $f:V\ra W$ is a morphism of left degree
$\sigma$ if $f(V_g)\subseteq W_{\sigma g}$ for any $g\in G$; this
means that $f$ is a morphism of graded vector spaces when regarded
as $f:V\ra (\sigma)W$. Similarly, $f$ is a morphism of right
degree $\sigma$ if $f(V_g)\subseteq W_{g\sigma}$ for any $g\in G$.

If $\mathcal{F}=(V,(V_{\alpha})_{\alpha \in \mathcal{C}})$  is a
$G$-graded $\rho$-flag and $\sigma\in G$, then the right
 suspension of $\mathcal{F}$ is
$\mathcal{F}(\sigma)=(V(\sigma),(V_{\alpha})_{\alpha \in
\mathcal{C}})$, and the left
 suspension   of $\mathcal{F}$ is
$(\sigma)\mathcal{F}=((\sigma)V,(V_{\alpha})_{\alpha \in
\mathcal{C}})$. It is clear that ${\rm END}(\mathcal{F})_\sigma$ is just
the space of morphisms of graded flags from $\mathcal{F}$ to
$(\sigma)\mathcal{F}$.

Let $\mathcal{F}=(V,(V_{\alpha})_{\alpha \in \mathcal{C}})$  be a
$G$-graded $\rho$-flag, with a homogeneous basis $B=\displaystyle
\bigcup_{\alpha \in \mathcal{C}}B_{\alpha}$ of $V$ providing the
flag structure.

Let $\mathcal{C}=\mathcal{C}^1 \cup \ldots \cup \mathcal{C}^q$ be
the decomposition of $\mathcal{C}$ in disjoint connected
components; these correspond to the connected components of the
undirected graph $\Gamma^u$. For each $1\leq t\leq q$, let
$\rho_t$ be the preorder relation on the set $\displaystyle
\bigcup_{\alpha \in {\mathcal C}^t}\alpha$, by restricting $\rho$.

If $V^t=\displaystyle \sum_{\alpha \in {\mathcal C}^t}V_\alpha$,
then $\mathcal{F}^t=(V^t,(V_{\alpha})_{\alpha \in \mathcal{C}^t})$
is a $G$-graded $\rho_t$-flag with basis $\displaystyle
\bigcup_{\alpha \in \mathcal{C}^t}B_{\alpha}$. Obviously,
$V=\displaystyle \bigoplus_{1\leq t\leq q}V^t$. In a formal way we
can write $\mathcal{F}=\mathcal{F}^1\oplus \ldots \oplus
\mathcal{F}^q$, where $\mathcal{F}$ is a $G$-graded $\rho$-flag,
and $\mathcal{F}^t$ is a $G$-graded $\rho_t$-flag for each $1\leq
t\leq q$. \\

\begin{definition} \label{giso}
Let $\rho$ and $\mu$ be isomorphic preorder relations (i.e. the
preordered sets on which $\rho$ and $\mu$ are defined are
isomorphic). Let $\mathcal C$ and $\mathcal D$ be the posets
associated with $\rho$ and $\mu$, and let $g:{\mathcal C}\ra
{\mathcal D}$ be an isomorphism of posets. We say that a
$\rho$-flag ${\mathcal F}=(V,(V_{\alpha})_{\alpha \in {\mathcal
C}}))$ is $g$-isomorphic to a $\mu$-flag ${\mathcal
G}=(W,(W_{\beta})_{\beta \in {\mathcal D}}))$ if there is a linear
isomorphism $u:V\ra W$ such that $u(V_\alpha)=W_{g(\alpha)}$ for
any $\alpha \in {\mathcal C}$. If $\mathcal F$ and $\mathcal G$
are $G$-graded flags, we say that they are $g$-isomorphic as
graded flags if there is such an $u$ which is a morphism of graded
vector spaces.
\end{definition}

\begin{lemma} \label{lemaisoEND}
With notation as in Definition \ref{giso}, if the $G$-graded flags
$\mathcal F$ and $\mathcal G$ are $g$-isomorphic, then
${\rm END}(\mathcal{F})$ and ${\rm END}(\mathcal{G})$ are isomorphic as
$G$-graded algebras.
\end{lemma}
\begin{proof}
If $u:V\ra W$ is a $g$-isomorphism between the graded flags
$\mathcal F$ and $\mathcal G$, then $\Phi:{\rm END}(\mathcal{F})\ra
{\rm END}(\mathcal{G})$, $\Phi (\phi)=u\phi u^{-1}$ is an isomorphism of
$G$-graded algebras.
\end{proof}

Now we will consider another $G$-graded $\rho$-flag
$\mathcal{F}'=(V',(V'_{\alpha})_{\alpha \in \mathcal{C}})$. As we
did for $\mathcal F$, we also have $V'=\displaystyle
\bigoplus_{1\leq t\leq q}V'^t$ and
$\mathcal{F}'=\mathcal{F}'^1\oplus \ldots \oplus \mathcal{F}'^q$,
where $\mathcal{F}'^t$ is a $G$-graded $\rho_t$-flag for each
$1\leq t\leq q$.

\begin{theorem}  \label{teoremaisograd}
Let $\mathcal{F}=(V,(V_{\alpha})_{\alpha \in \mathcal{C}})$ and
$\mathcal{F}'=(V',(V'_{\alpha})_{\alpha \in \mathcal{C}})$ be
$G$-graded $\rho$-flags. Then the following assertions are
equivalent:\\
{\rm (1)} ${\rm END}(\mathcal{F})$ and ${\rm END}(\mathcal{F}')$ are isomorphic as
$G$-graded algebras.\\
{\rm (2)} There exist $g\in {\rm Aut}_0(\mathcal{C})$,
$\sigma_1,\ldots,\sigma_q\in G$ and a $g$-isomorphism $\gamma:V\ra
V'$ between the (ungraded) $\rho$-flags $\mathcal{F}$ and
$\mathcal{F}'$, such that
$\gamma_{|V^t}^{|V'^{\overline{g}(t)}}:V^t\ra
V'^{\overline{g}(t)}$ is a linear isomorphism of left degree
$\sigma_t$ for any $1\leq t\leq q$, where $\overline{g}\in S_q$ is
the permutation induced by $g$, i.e.
$g(\mathcal{C}^t)=\mathcal{C}^{\overline{g}(t)}$.\\
{\rm (3)} There exists a permutation $\tau\in S_q$, an isomorphism
$g_t:\mathcal{C}^t\ra \mathcal{C}^{\tau(t)}$ for each $1\leq t\leq
q$, and $\sigma_1,\ldots ,\sigma_q\in G$, such that
$\mathcal{F}^t(\sigma_t)$ is $g_t$-isomorphic to
$\mathcal{F}'^{\tau(t)}$ for any $1\leq t\leq q$.
\end{theorem}
\begin{proof}
(1)$\Rightarrow $(2) Let $\phi:{\rm END}(\mathcal{F})\ra
{\rm END}(\mathcal{F}')$ be an isomorphism of $G$-graded algebras. We
follow the line of proof of Proposition \ref{surjective} and its
notation, adding the additional information related to the graded
structure. Let $g_i={\rm deg}\, v_i$ for any $i$. Then $E_{ij}$ is
a homogeneous element of degree $g_ig_j^{-1}$ of
${\rm END}(\mathcal{F})$, and $F_{ij}=\phi (E_{ij})$ also has degree
$g_ig_j^{-1}$ in ${\rm END}(\mathcal{F}')$.

Now $Q_i={\rm Im}\, F_{ii}$ is a graded subspace of $V'$,
$V'=\displaystyle \bigoplus_{1\leq i\leq n}Q_i$ and each $Q_i$ is
1-dimensional. For each $i$ pick $w_i\in Q_i\setminus \{0\}$,
which is a homogeneous element. Then for any $i,j$ with $i\rho j$
we have $F_{ij}(w_j)=a_{ij}w_i$ for some $a_{ij}\in k^*$, and
$a_{ij}a_{jr}=a_{ir}$ for any $i,j,r$ with $i\rho j$ and $j\rho
r$.

Since $(F_{ij})_{|Q_j}^{|Q_i}:Q_j\ra Q_i$ is a linear isomorphism
of degree $g_ig_j^{-1}$ (with inverse
$(F_{ji})_{|Q_i}^{|Q_j}:Q_i\ra Q_j$), we obtain that $Q_i\simeq
(g_jg_i^{-1})Q_j\simeq (g_i^{-1})(g_j)Q_j$ for any $i,j$ with
$i\rho j$, so then $(g_i)Q_i\simeq (g_j)Q_j$ as graded vector
spaces. This implies that $(g_i)Q_i$ has the same isomorphism type
for all $i\in\alpha$ with $\alpha$ lying in a connected component
of $\mathcal{C}$.

On the other hand, if $R_i=kv_i$ for any $i$, then $(g_i)R_i$ and
$(g_j)R_j$ are isomorphic graded vector spaces for any $i,j$. Then
there are $\sigma_1,\ldots,\sigma_q\in G$ such that for any $1\leq
t\leq q$ we have $Q_i\simeq R_i(\sigma_t)$ for any $i\in \alpha$
with $\alpha \in \mathcal{C}^t$. Indeed, fix some $t$ and pick
$\alpha_0\in\mathcal{C}^t$ and $i_0\in \alpha_0$. Since $Q_{i_0}$
and $R_{i_0}$ are 1-dimensional graded vector spaces, there exists
$\sigma_t\in G$ such that $Q_{i_0}(\sigma_t)\simeq R_{i_0}$. Then
for any $\alpha \in \mathcal{C}^t$ and any $i\in \alpha$ one has
$$Q_i(\sigma_t)\simeq ((g_{i_0}g_i^{-1})Q_{i_0})(\sigma_t)=(g_{i_0}g_i^{-1})(Q_{i_0}(\sigma_t))\simeq
(g_{i_0}g_i^{-1})R_{i_0}\simeq R_i$$ Thus we obtain that
\begin{equation} \label{gradebaza} {\rm deg}\, v_i=({\rm deg}\,
w_i)\sigma_t \;\;\; \mbox{for any }i\in\alpha \mbox{ with
}\alpha\in \mathcal{C}^t
\end{equation}

As in the proof of Proposition \ref{surjective}, the linear map
$\gamma:V\ra V'$ with $\gamma(v_i)=w_i$ for any $i$, is a
$\phi'$-isomorphism for a certain algebra isomorphism $\phi':{\rm
{\rm END}}(\mathcal{F})\ra {\rm END}(\mathcal{F}')$, and then there
exists $g\in {\rm Aut}_0(\mathcal{C})$ with
$\gamma(V_\alpha)=V'_{g(\alpha)}$ for any $\alpha \in
\mathcal{C}$. By (\ref{gradebaza}), $\gamma_{|V^t}$ is a linear
morphism of left degree $\sigma_t$, and we are done.\\

(2)$\Rightarrow$(3) Take $\tau=\overline{g}$, and let
$g_t:\mathcal{C}^t\ra\mathcal{C}^{\tau(t)}$ be the restriction and
corestriction of $g$ for each $t$. Then the restriction and
corestriction of $\gamma$ to $V^t$ and $V'^{\tau(t)}$ gives a
$g_t$-isomorphism of graded flags $\mathcal{F}^t(\sigma_t)\simeq
\mathcal{F}'^{\tau(t)}$.\\

(3)$\Rightarrow$(1) It is clear that ${\rm END}(\mathcal{F})\simeq
{\rm END}(\mathcal{F}^1)\times \ldots \times
{\rm END}(\mathcal{F}^q)$ and ${\rm END}(\mathcal{F}')\simeq
{\rm END}(\mathcal{F'}^1)\times \ldots \times
{\rm END}(\mathcal{F'}^q)$ as $G$-graded algebras. By Lemma
\ref{lemaisoEND}, we see that $
{\rm END}(\mathcal{F}^t(\sigma_t))\simeq
{\rm END}(\mathcal{F}'^{\tau(t)})$ for any $1\leq t\leq q$. As it is
obvious that ${\rm END}(\mathcal{F}^t(\sigma_t))=
{\rm END}(\mathcal{F}^t)$, we get ${\rm END}(\mathcal{F}^t)\simeq
{\rm END}(\mathcal{F}'^{\tau(t)})$ for any $1\leq t\leq q$. We conclude
that ${\rm END}(\mathcal{F})\simeq {\rm END}(\mathcal{F}')$.
\end{proof}

\section{Classification of gradings arising from graded
flags}\label{sectionorbits}

The aim of this section is to classify $G$-gradings on $M(\rho,
k)$ arising from graded flags by the orbits of a certain group
action. We keep all the notations of Section
\ref{sectionisograded}. We first consider three group actions on
the set $G^n$.\\

$\bullet$ ${\rm Aut}_0(\mathcal{C})$ acts to the right on $G^n$ by
$$(h_i)_{1\leq i\leq n}\leftarrow g=(h_{\tilde{g}(i)})_{1\leq i\leq
n}$$ for any $(h_i)_{1\leq i\leq n}\in G^n$ and $g\in
{\rm Aut}_0(\mathcal{C})$.\\

$\bullet$ $G^q$ acts to the right on $G^n$ by $$(h_i)_{1\leq i\leq
n}\leftarrow (\sigma_t)_{1\leq t\leq q}=(h'_i)_{1\leq i\leq n}$$
where for each $i$ we define $h'_i=h_i\sigma_p$, where $p$ is such
that $\hat{i}\in \mathcal{C}^p$.\\

$\bullet$ For each $\alpha \in \mathcal{C}$ let $S(\alpha)$ be the
symmetric group of $\alpha$ (regarded as a subset of $\{
1,\ldots,n\}$. We consider the group $\prod_{\alpha\in
\mathcal{C}}S(\alpha)$, which is a Young subgroup of $S_n$
(isomorphic to $\prod_{\alpha \in \mathcal{C}}S_{m_\alpha}$). Then
$\prod_{\alpha\in \mathcal{C}}S(\alpha)$ acts to the right on
$G^n$ by $$(h_i)_{1\leq i\leq n}\leftarrow (\psi_\alpha)_{\alpha
\in \mathcal{C}}=(h'_i)_{1\leq i\leq n}$$ with $h'_i$ defined by
$h'_i=h_{\psi_\alpha(i)}$, where $\alpha=\hat{i}$, for each $i$.\\

Now there is a right action of the group ${\rm Aut}_0(\mathcal{C})$ on
the group $G^q$ defined by $$(\sigma_t)_{1\leq t\leq q}\leftarrow
g=(\sigma_{\tau(t)})_{1\leq t\leq q}$$ where $\tau \in S_q$ is the
permutation induced by $g$. Then we have a right semidirect
product ${\rm Aut}_0(\mathcal{C})\ltimes G^q$. Moreover, the
compatibility condition (\ref{comprightactionset}) holds for these
actions, i.e.
\begin{equation} \label{comp12}
((h_i)_{1\leq i\leq n}\leftarrow (\sigma_t)_{1\leq t\leq
q})\leftarrow g=((h_i)_{1\leq i\leq n}\leftarrow
g)\leftarrow((\sigma_t)_{1\leq t\leq q}\leftarrow g)
\end{equation}
Indeed, it is easy to see that both sides of equation
(\ref{comp12}) have on the $i$th position
$h_{\tilde{g}(i)}\sigma_{\tau (p)}$, where $p$ is such that
$\hat{i}\in \mathcal{C}^p$. We conclude that
${\rm Aut}_0(\mathcal{C})\ltimes G^q$ acts to the right on the set $G^n$
by
$$(h_i)_{1\leq i\leq n}\leftarrow (g\ltimes (\sigma_t)_{1\leq t\leq
q})=((h_i)_{1\leq i\leq n}\leftarrow g)\leftarrow
(\sigma_t)_{1\leq t\leq q}$$

On the other hand, it is  straightforward  to check that
${\rm Aut}_0(\mathcal{C})$ acts to the left on the group
$\prod_{\alpha\in \mathcal{C}}S(\alpha)$ by $g\ra
(\psi_\alpha)_{\alpha \in \mathcal{C}}= (\psi'_\alpha)_{\alpha \in
\mathcal{C}}$, where for any $\alpha\in \mathcal{C}$,
$\psi'_{\alpha}$ is defined by
$$\psi'_{\alpha}(i)=\tilde{g}(\psi_{g^{-1}(\alpha)}(\tilde{g}^{-1}(i)))$$
for any $i\in \alpha$. Moreover, one can check that

\begin{equation} \label{comp13}
((h_i)_{1\leq i\leq n}\leftarrow g)\leftarrow
(\psi_\alpha)_{\alpha \in \mathcal{C}}=((h_i)_{1\leq i\leq
n}\leftarrow (g\ra (\psi_\alpha)_{\alpha \in
\mathcal{C}}))\leftarrow g
\end{equation}

The left action of ${\rm Aut}_0(\mathcal{C})$  on $\prod_{\alpha\in
\mathcal{C}}S(\alpha)$ induces a left action of
${\rm Aut}_0(\mathcal{C})\ltimes G^q$  on  $\prod_{\alpha\in
\mathcal{C}}S(\alpha)$, with $1\ltimes G^q$ acting trivially on
$\prod_{\alpha\in \mathcal{C}}S(\alpha)$, and then we can consider
the left semidirect product $\prod_{\alpha\in
\mathcal{C}}S(\alpha)\rtimes ({\rm Aut}_0(\mathcal{C})\ltimes G^q)$.

The right actions of $\prod_{\alpha\in \mathcal{C}}S(\alpha)$ and
$G^q$ on $G^n$ commute, i.e.
\begin{equation} \label{comp23}
((h_i)_{1\leq i\leq n}\leftarrow (\psi_\alpha)_{\alpha \in
\mathcal{C}})\leftarrow (\sigma_t)_{1\leq t\leq q}=((h_i)_{1\leq
i\leq n}\leftarrow (\sigma_t)_{1\leq t\leq q})\leftarrow
(\psi_\alpha)_{\alpha \in \mathcal{C}} \end{equation} Indeed, it
is easily checked that on the $i$th position of each side one
finds the element $h_{\psi_{\alpha}(i)}\sigma_p$, where
$\alpha=\hat{i}\in \mathcal{C}^p$. Now (\ref{comp13}) and
(\ref{comp23}) show that the compatibility relation
(\ref{comprightactionset2}) is satisfied for the right actions of
${\rm Aut}_0(\mathcal{C})\ltimes G^q$ and $\prod_{\alpha\in
\mathcal{C}}S(\alpha)$ on $G^n$. In conclusion, the group
$\prod_{\alpha\in \mathcal{C}}S(\alpha)\rtimes
({\rm Aut}_0(\mathcal{C})\ltimes G^q)$ acts to the right on the set
$G^n$ by
$$(h_i)_{1\leq i\leq n}\leftarrow ((\psi_\alpha)_{\alpha \in
\mathcal{C}}\rtimes (g\ltimes (\sigma_t)_{1\leq t\leq
q}))=(((h_i)_{1\leq i\leq n}\leftarrow (\psi_\alpha)_{\alpha \in
\mathcal{C}})\leftarrow g)\leftarrow (\sigma_t)_{1\leq t\leq q}.$$

Now we can prove the following.

\begin{theorem}
The isomorphism types of $G$-gradings of the type $
{\rm END}(\mathcal{F})$, where $\mathcal{F}$ is a $G$-graded
$\rho$-flag, are classified by the orbits of the right action of
the group $\prod_{\alpha\in \mathcal{C}}S(\alpha)\rtimes
({\rm Aut}_0(\mathcal{C})\ltimes G^q)$  on the set $G^n$.
\end{theorem}
\begin{proof}
We first need some simple remarks. If $V$ and $V'$ are $G$-graded
vector spaces with homogeneous bases $\{ v_1,\ldots,v_n\}$,
respectively $\{ v'_1,\ldots,v'_n\}$, then $V$ and $V'$ are
isomorphic as $G$-graded vector spaces if and only if $\text{dim }\,
V_g=\text{ dim }\, V'_g$ for any $g\in G$, and this is also
equivalent to the fact that the $n$-tuple $({\rm deg}\, v_1,\ldots
,{\rm deg}\, v_n)$ of elements of $G$ is obtained from $({\rm
deg}\, v'_1,\ldots ,{\rm deg}\, v'_n)$ by a permutation.

Also, if $\sigma\in G$, then $V(\sigma)\simeq V'$ if and only if
$(({\rm deg}\, v_1)\sigma^{-1},\ldots ,({\rm deg}\,
v_n)\sigma^{-1})$ is obtained from $({\rm deg}\, v'_1,\ldots ,{\rm
deg}\, v'_n)$ by a permutation; this follows from the fact that a
homogeneous element of degree $g$ of $V$ has degree $g\sigma^{-1}$
in $V(\sigma)$.

Let $\mathcal{F}$ be a $\rho$-flag with basis $B=\{ v_1,\ldots
,v_n\}$ as in Definition \ref{de}. A $G$-graded structure on
$\mathcal{F}$ is given by assigning arbitrary degrees $h_1,\ldots
,h_n\in G$ to $v_1,\ldots ,v_n$. Thus $G$-graded structures on
$\mathcal{F}$ are given by elements $(h_1,\ldots ,h_n)$ of $G^n$.

Let $\mathcal{F}$ and $\mathcal{F}'$ be $G$-graded $\rho$-flags
given by $n$-tuples $(h_1,\ldots ,h_n)$ and $(h'_1,\ldots ,h'_n)$
as above. The homogeneous basis of $V$ and $V'$ are denoted by
$B=\displaystyle \bigcup_{\alpha \in \mathcal{C}} B_{\alpha}$ and
$B'=\displaystyle \bigcup_{\alpha \in \mathcal{C}} B'_{\alpha}$.

We show that ${\rm END}(\mathcal{F})\simeq {\rm END}(\mathcal{F}')$ as
$G$-graded algebras if and only if $(h_1,\ldots ,h_n)$ and
$(h'_1,\ldots ,h'_n)$ are in the same orbit of $G^n$ with respect
to the right action of $\prod_{\alpha\in
\mathcal{C}}S(\alpha)\rtimes ({\rm Aut}_0(\mathcal{C})\ltimes G^q)$, and
this will finish the proof.

By Theorem \ref{teoremaisograd}, ${\rm END}(\mathcal{F})\simeq
{\rm END}(\mathcal{F}')$ if and only if there exists $g\in
{\rm Aut}_0(\mathcal{C})$ such that $\mathcal{F}^t(\sigma_t)$ is
$g_t$-isomorphic as a graded flag to
$\mathcal{F}'^{\overline{g}(t)}$ for any $1\leq t\leq q$, where
$\overline{g} \in S_q$ is the permutation such that for any $t$,
$g(\mathcal{C}^t)=\mathcal{C}^{\overline{g}(t)}$, and
$g_t:\mathcal{C}^t\ra \mathcal{C}^{\overline{g}(t)}$ is the
isomorphism of posets induced by $g$ via restriction and
corestriction. Using the considerations above, we obtain that
${\rm END}(\mathcal{F})\simeq {\rm END}(\mathcal{F}')$ if and only if there
exists $g\in {\rm Aut}_0(\mathcal{C})$ such that for each $1\leq t\leq
q$ and any $\alpha \in \mathcal{C}^t$, the degrees of the elements
of $B_\alpha$ multiplied to the right by $\sigma_t^{-1}$ are
obtained by a permutation from the elements of $B'_{g(\alpha)}$.
But this is equivalent to
$$((h_i)_{1\leq i\leq n}\leftarrow (\sigma_t^{-1})_{1\leq t\leq
q})\leftarrow (\psi_\alpha)_{\alpha \in \mathcal{C}}=(h'_i)_{1\leq
i\leq n}\leftarrow g$$ for some $(\psi_\alpha)_{\alpha \in
\mathcal{C}}\in \prod_{\alpha\in \mathcal{C}}S(\alpha)$. Since the
right actions of $G^q$ and $\prod_{\alpha\in
\mathcal{C}}S(\alpha)$ on $G^n$ commute, this is the same with
$$(h_i)_{1\leq i\leq n}= (((h'_i)_{1\leq
i\leq n}\leftarrow g)\leftarrow (\sigma_t)_{1\leq t\leq
q})\leftarrow (\psi_\alpha^{-1})_{\alpha \in \mathcal{C}}$$ which
is the same to $(h_i)_{1\leq i\leq n}$ and $(h'_i)_{1\leq i\leq
n}$ lying in the same orbit of the right action of
$\prod_{\alpha\in \mathcal{C}}S(\alpha)\rtimes
({\rm Aut}_0(\mathcal{C})\ltimes G^q)$  on $G^n$.
\end{proof}

\begin{example}
As particular cases of our results we obtain the following.\\
(1) Let $A=M_n(k)$, the full matrix algebra. If $G$ is a group,
then the good $G$-gradings on $A$ are all isomorphic to gradings
of the form ${\rm END}(V)$, where $V$ is a $G$-graded vector space
of dimension $n$. These gradings are classified by the orbits of
the biaction of the groups $S_n$ (by usual permutations of the
elements) and $G$ (by right translations) on $G^n$. Indeed, in
this case $\mathcal C$ is a singleton, so obviously ${\rm
{\rm Aut}}_0({\mathcal C})$ is trivial. This result appears in
\cite{cdn}.\\
(2) Let $A$ be the algebra of upper block triangular matrices of
type $m_1,\ldots ,m_r$, where $n=m_1+\ldots +m_r$. Then $\mathcal
C$ is isomorphic to the poset $\{1,\ldots ,r\}$ with the usual
order, so any good grading on $A$ is of the type $ {\rm
END}({\mathcal F})$, where $\mathcal F$ is a graded (usual) flag
of signature $(m_1,\ldots ,m_r)$. Again, ${\rm Aut}_0({\mathcal
C})$ is trivial and $\mathcal C$ is connected, so the isomorphism
types of good $G$-gradings on $A$ are classified by the orbits of
the biaction of a Young subgroup $S_{m_1}\times \ldots \times
S_{m_r}$ (by permutations) and $G$ (by translations) on $G^n$.
This result appears in \cite{bd}.
\end{example}


\begin{thebibliography}{99}

\bibitem{bsz} Yu. A. Bahturin, S. K. Sehgal, M. V. Zaicev,  Group
gradings on associative algebras, J. Algebra \textbf{241} (2001),
677--698.

\bibitem{bz} Yu. A. Bahturin, M. V. Zaicev, Group
gradings on matrix algebras, Canad. Math. Bull. \textbf{45}
(2002), 499--508.

\bibitem{bd} M. B\u ar\u ascu, S. D\u asc\u alescu, Good gradings
on upper block triangular matrix algebras, Comm. Algebra {\bf 41}
(2013), 4290-4298.

\bibitem{cdn} S. Caenepeel, S. D\u asc\u alescu, C. N\u ast\u asescu,  On Gradings of Matrix Algebras and Descent
Theory, Comm. Algebra \textbf{30} (2002), 5901--5920.


\bibitem{coelho} S. P. Coelho, The automorphism group of a
structural matrix algebra, Linear Alg. Appl. {\bf 195} (1993),
35-58.

\bibitem{lb} V. Lakshmibai and J. Brown, Flag varieties. An
interplay of geometry, combinatorics and representation theory,
Hindustan Book Agency, 2009.

\bibitem{nvo} C. N\u{a}st\u{a}sescu and F. van Oystaeyen, Methods
of graded rings, Lecture Notes in Math., vol. 1836 (2004),
Springer Verlag.

\bibitem{nowicki} A. Nowicki, Derivations of special subrings of
matrix rings and regular graphs, Tsukuba J. Math. {\bf 7} (1983),
281-297.


\bibitem{spiegel} E. Spiegel and C. J. O'Donnell, Incidence
algebras, Pure and Appl. Math. {\bf 206} (1997), Marcel Dekker,
New York.

\bibitem{vW} L. Van Wyk, Maximal left ideals in structural matrix
rings, Comm. Algebra {\bf 16} (1988), 399-419.

\end{thebibliography}
\end{document}